%% file: Thread_Quiver.tex
\documentclass{amsart}

\usepackage{graphicx}
\usepackage{thmtools, thm-restate}
\usepackage{a4wide}
\usepackage{psfrag}
\usepackage{hyperref}
\usepackage[all]{xy}
\xyoption{2cell}
\UseAllTwocells

\let\cal\mathcal
\def\AA{{\cal A}}

\def\CC{{\cal C}}
\def\DD{{\cal D}}

\def\FF{{\cal F}}

\def\II{{\cal I}}
\def\JJ{{\cal J}}

\def\LL{{\cal L}}
\def\MM{{\cal M}}

\def\OO{{\cal O}}
\def\PP{{\cal P}}
\def\QQ{{\cal Q}}

\def\SS{{\cal S}}
\def\TT{{\cal T}}

\let\blb\mathbb

\def\bP{{\blb P}}

\def\bZ{{\blb Z}}

\def\bN{{\blb N}}

\def\bS{{\blb S}}

\def\bZ{{\blb Z}}

\let\frak\mathfrak

\def\aa{\frak{a}}
\def\bb{\frak{b}}
\def\cc{\frak{c}}

\def\ff{\frak{f}}
\def\gg{\frak{g}}
\def\pp{\frak{p}}

\def\Mod{\operatorname{Mod}}

\def\Modlfd{\operatorname{\Mod_{\text{lfd}}}}
\def\mod{\operatorname{mod}}
\def\modfd{\operatorname{mod_{\text{fd}}}}

\def\coh{\mathop{\text{\upshape{coh}}}}

\def\rad{\operatorname {rad}}

\def\Rep{\operatorname {Rep}}

\def\rep{\operatorname{rep}}

\def\Ext{\operatorname {Ext}}

\def\Hom{\operatorname {Hom}}

\def\supp{\operatorname {supp}}

\def\im{\operatorname {im}}

\def\ker{\operatorname {ker}}
\def\Ker{\operatorname {ker}}
\def\Tor{\operatorname {Tor}}

\def\r{\rightarrow}

\DeclareMathOperator{\Irr}{Irr}
\DeclareMathOperator{\Mor}{Mor}
\DeclareMathOperator{\colim}{colim}

\DeclareMathOperator{\ind}{ind}

\newcommand\Db{D^{b}}
\newcommand\Kb{K^{b}}

\renewcommand\r{d^{\bullet}}

\newtheorem{lemma}{Lemma}[section]
\newtheorem{proposition}[lemma]{Proposition}
\newtheorem{theorem}[lemma]{Theorem}
\newtheorem{corollary}[lemma]{Corollary}

\newcounter{MyCounter}

\theoremstyle{definition}

\newtheorem{example}[lemma]{Example}
\newtheorem{definition}[lemma]{Definition}

\newtheorem{construction}[lemma]{Construction}

\newtheorem{observation}[lemma]{Observation}
{

}

\theoremstyle{remark}

\newtheorem{remark}[lemma]{Remark}

\newdimen\uboxsep \uboxsep=1ex
\def\uboxn#1{\vtop to 0pt{\hrule height 0pt depth 0pt\vskip\uboxsep
\hbox to 0pt{\hss #1\hss}\vss}}

\def\uboxs#1{\vbox to 0pt{\vss\hbox to 0pt{\hss #1\hss}
\vskip\uboxsep\hrule height 0pt depth 0pt}}

\def\Ob{\operatorname{Ob}}

\def\Free{\operatorname{Free}}
\def\CQ{\mathfrak{CondQuiver}}
\def\Cat{\mathfrak{Cat}}

\def\kVar{\mbox{$k$-$\mathfrak{Var}$}}

\def\CAT{\mathfrak{CAT}}

\newcommand\exa{\nopagebreak \begin{center}\smallskip \nopagebreak               \begin{minipage}[t]{0.5\textwidth}\sloppy}
\newcommand\exb{\end{minipage}\begin{minipage}[t]{0.5\textwidth}\sloppy}
\newcommand\exc{\end{minipage}\smallskip\end{center}}

\title{Representations of Thread Quivers}
\author{Carl Fredrik Berg}
\author{Adam-Christiaan van Roosmalen}
\address{Carl Fredrik Berg\\Institutt for matematiske fag\\
NTNU\\7491 Trondheim\\Norway\\
(Currently working for Statoil R\&D Centre\\Arkitekt Ebbells veg 10\\Rotvoll\\7053 Trondheim\\Norway)} \email{carlpaatur@hotmail.com}
\address{Adam-Christiaan van Roosmalen\\Fakult\"at f\"ur Mathematik\\ Universit\"at Bielefeld\\D-33501 Bielefeld\\Germany}\email{vroosmal@math.uni-bielefeld.de}


\begin{document}

\bibliographystyle{amsplain}

\begin{abstract}
We introduce thread quivers as an (infinite) generalization of quivers, and show that every $k$-linear ($k$ algebraically closed) hereditary category with Serre duality and enough projectives is equivalent to the category of finitely presented representations of a thread quiver.  In this way, we obtain an explicit construction of a new class of hereditary categories with Serre duality.
\end{abstract}

\maketitle

\setcounter{tocdepth}{1}
\tableofcontents

\input{Introduction}
\input{SemiHereditary}
\input{Appendix}

\providecommand{\bysame}{\leavevmode\hbox to3em{\hrulefill}\thinspace}
\providecommand{\MR}{\relax\ifhmode\unskip\space\fi MR }
\providecommand{\MRhref}[2]{%
  \href{http://www.ams.org/mathscinet-getitem?mr=#1}{#2}
}
\providecommand{\href}[2]{#2}

\end{document}

%% file: Introduction.tex
\section{Introduction}

We will assume all categories are $k$-linear for an algebraically closed field $k$.  In representation theory, one often considers the category of finite-dimensional $k$-representations of a finite acyclic quiver.  The category of representations is then an Ext-finite hereditary category with Serre duality and enough projectives.

One possible generalization is to replace the `finite quiver' with a `strongly locally finite quiver', i.e. a quiver such that all the indecomposable projective and injective representations have finite $k$-dimension.  Again, the category of finite-dimensional $k$-representations of a strongly locally finite quiver is an Ext-finite hereditary category with Serre duality and enough projectives.


In this paper, we provide a further generalization by replacing the strongly locally finite quiver with a strongly locally finite \emph{thread quiver}.  Roughly speaking, a strongly locally finite thread quiver can be thought of as a strongly locally finite quiver where some arrows have been replaced with locally discrete linearly ordered sets (thus combining quivers with the posets considered in \cite{Ringel02}).  We will give a more precise definition of a thread quiver further in this introduction.

The following theorem is our main result (it will be proven in Section \ref{section:Representations}).

\begin{restatable}{maintheorem}{theoremB}\label{theorem:EnoughProjectives}
\begin{enumerate}
\item The category of finitely presented representations of a strongly locally finite thread quiver is a hereditary Ext-finite category with Serre duality and enough projectives.
\item A hereditary Ext-finite category with Serre duality and enough projectives is equivalent to the category of finitely presented representations of a strongly locally finite thread quiver.
\end{enumerate}
\end{restatable}



Our interest in thread quivers comes from a project to understand and classify hereditary category satisfying some additional homological properties.  Highlights of this project include Happel's classification of hereditary categories having a tilting object (\cite{Happel01}) and Reiten-Van den Bergh's classification of noetherian hereditary categories with Serre duality (\cite{ReVdB02}).

In the latter, a new type of hereditary category with Serre duality was encountered, which is generated by preprojectives but does not have enough projectives; these categories were constructed by formally inverting a right Serre functor of a related category (we refer to \cite{ReVdB02} for details).  In \cite{Ringel02b}, Ringel gave an alternative construction using ray quivers.  In \cite{BergVanRoosmalen08}, it was shown that every noetherian category generated by preprojective objects is derived equivalent to the category of finite-dimensional representations of a strongly locally finite quiver.

When considering hereditary categories which have Serre duality, but which are not necessarily noetherian, Reiten suggests in \cite{Reiten02} to consider categories generated by preprojective objects as a possible first step.  Generalizing the techniques from \cite{BergVanRoosmalen08}, we have shown in \cite{BergVanRoosmalen10} that every such category is derived equivalent to a hereditary category with Serre duality and enough projectives.  Theorem \ref{theorem:EnoughProjectives} classifies these categories by thread quivers, completing the project proposed in \cite{Reiten02} up to derived equivalence.

Thread quivers will be used to classify a certain kind of additive categories, and the modules over those categories will be the representations of the thread quiver.  To make this more precise, we recall some definitions.

A \emph{finite $k$-variety} is a Hom-finite additive category $\aa$ where idempotents split.  The functors $\aa(-,A)$ and $\aa(A,-)^*$ from $\aa$ to $\mod k$ will be called \emph{standard projective modules} and \emph{standard injective modules}, respectively.  We will write $\mod \aa$ for the category of contravariant functors $\aa \to \mod k$ which are finitely presentable by standard projectives.

We will say that a finite $k$-variety $\aa$ is \emph{dualizing} if and only if (Proposition \ref{proposition:Dualizing}) $\aa$ has pseudokernels and pseudocokernels (thus $\mod \aa$ and $\mod \aa^\circ$ are abelian, where $\aa^\circ$ is the dual category of $\aa$), every standard projective object is cofinitely generated by standard injectives, and every standard injective object is finitely generated by standard projectives.


The following theorem will be proven in Section \ref{section:Dualizing}.

\begin{restatable}{maintheorem}{theoremA}\label{theorem:RepSerreDuality}
Let $\aa$ be a finite $k$-variety.  The following are equivalent:
\begin{enumerate}
\item $\mod \aa$ is abelian and has Serre duality,
\item $\aa$ is a dualizing $k$-variety and every object of $\mod \aa$ has finite projective and finite injective dimension.
\end{enumerate}
\end{restatable}

A finite $k$-variety $\aa$ is called \emph{semi-hereditary} if and only if the category $\mod \aa$ is abelian and hereditary.  It has been shown (\cite{AuslanderReiten75, vanRoosmalen06}, see Proposition \ref{proposition:ShIsLocal}) that $\aa$ is semi-hereditary if and only if every full (preadditive) subcategory with finitely many elements is semi-hereditary.

Intuitively, a semi-hereditary dualizing $k$-variety can be seen as the free $k$-linear path category $kQ$ of a strongly locally finite quiver $k$ without relations, but where some arrows are replaced with infinite locally discrete (= without accumulation points) linearly ordered posets.  To accommodate this extra information, we will introduce thread quivers.

A \emph{thread quiver} consists of the following information:
\begin{itemize}
\item A quiver $Q=(Q_0,Q_1)$ where $Q_0$ is the set of vertices and $Q_1$ is the set of arrows.
\item A decomposition $Q_1 = Q_s \coprod Q_t$.  Arrows in $Q_s$ will be called \emph{standard arrows}, while arrows in $Q_t$ will be referred to as \emph{thread arrows}.
\item For every thread arrow $t$, there is an associated linearly ordered set $\PP_t$, possibly empty.
\end{itemize}

With every thread arrow $\xymatrix@1{t:x \ar@{..>}[r]^-{\PP_t} & y}$ in $Q$ we associate a locally discrete (= there are no accumulation points) linearly ordered poset $\LL_t$ with a minimal and a maximal element: namely $\LL_t = \bN \cdot (\PP_t \stackrel{\rightarrow}{\times}\bZ) \cdot -\bN$, where $A \cdot B$ means ``first $A$, then $B$'' and $\PP_t \stackrel{\rightarrow}{\times}\bZ$ is the poset $\PP_t \times \bZ$ with the lexicographical ordering (see Figure \ref{fig:Poset}).

\begin{figure}\label{fig:Poset}
$$\xymatrix@C=1pt@R=5pt{\bN & \cdot & \PP_t \stackrel{\rightarrow}{\times}\bZ & \cdot & -\bN \\
*+[F.]{0 \to 1 \to 2 \to \cdots} &\cdots &*+[F.]{\cdots \to (p,-1) \to (p,0) \to (p,1) \to \cdots}&\cdots & *+[F.]{\cdots \to -2 \to -1 \to -0} }$$
\caption{The poset $\LL_t$}
\end{figure}
The poset $\LL_t$ is interpreted as a category in the usual sense, and $k\LL_t$ will denote the associated $k$-linear additive category.

To recuperate the semi-hereditary dualizing $k$-variety from the thread quiver $Q$, we need to replace the thread arrows with the corresponding linearly ordered posets.  Since it is cumbersome --albeit possible-- to do this ``by hand'', we will prefer a more global approach: a 2-pushout.

For a thread quiver $Q$, denote by $Q_u$ the underlying ``regular'' quiver, thus where all the thread arrows are regular (and unlabeled) arrows.  By $kQ_u$ we will denote the normal $k$-linear additive path category of $Q_u$.

For every thread arrow $\xymatrix@1{t:x_t \ar@{..>}[r]^-{\PP_t} & y_t}$ of $Q$, there are associated functors $k(x_t \to y_t) \to kQ_u$ and $k(x_t \to y_t) \longrightarrow k\LL_t$ where $\LL_t = \bN \cdot (\PP_t \stackrel{\rightarrow}{\times}\bZ) \cdot -\bN$ as above.  The semi-hereditary dualizing $k$-variety $kQ$ is then defined as a 2-pushout of
$$\xymatrix{
\bigoplus_{t \in Q_t} k (\cdot \to \cdot) \ar[r]^-f \ar[d]_g & {k Q_u} \\
\bigoplus_{t \in Q_t} k\LL_t
}$$
effectively replacing the thread arrows in $Q$ with the required linearly ordered posets.  It will be shown in Theorem \ref{theorem:ThreadQuivers} that every semi-hereditary dualizing $k$-variety is obtained in this way.

%


An added advantage of the construction of $kQ$ as a 2-pushout is that we gain the following description of the category $\rep_k Q = \mod kQ$: the objects of the category $\rep Q$ are given by the following data
\begin{enumerate}
\item a finitely presented representation $N(-):kQ_u \to \mod k$ of $kQ_u$,
\item for every thread $t$, a finitely presented representation $L_t(-):\LL_t \to \mod k$, and
\item a natural equivalence $\alpha: \oplus_t L_t(g-) \Rightarrow N(f-)$.
\end{enumerate}
The morphisms are given by the modifications, thus given the data $(N,\{L_t\}_t, \alpha)$ and $(N',\{L'_t\}_t, \alpha')$ of two representations, a morphism is given by
\begin{enumerate}
\item a natural transformation $\beta: N \Rightarrow N'$,
\item a natural transformation $\gamma: \oplus_t L_t \Rightarrow \oplus_t L'_t$
\end{enumerate}
such that the following diagram commutes:
$$\xymatrix{
\oplus_t L_t(g-) \ar[r]^-\alpha \ar[d]_{\gamma \bullet 1_g} & N(f-) \ar[d]^{\beta \bullet 1_f}\\
\oplus_t L'_t(g-) \ar[r]_-{\alpha'} & N'(f-)
}$$

{\bf Acknowledgments} The authors would like to thank Idun Reiten, Sverre Smal\o, and Michel Van den Bergh for their many useful discussions and helpful ideas.  We would also like to thank Tobias Fritz for pointing us to \cite{Waschkies04}, and Catharina Stroppel for her useful comments on an earlier version of this article.  The second author gratefully acknowledges the hospitality and support of the Max-Planck-Institut f\"{u}r Mathematik in Bonn, the Norwegian University of Science and Technology, and the University of Regina.

%% file: SemiHereditary.tex
\renewcommand{\r}{\to}  

\section{Preliminaries}

\subsection{Notations and conventions}

Throughout, let $k$ be an algebraically closed field.  All considered categories will be assumed to be $k$-linear unless explicitly mentioned.  When $V$ is a $k$-vector space, then $V^*$ denotes the dual vector space, thus $V^* = \Hom_k(V,k)$.

When $\PP_1, \leq_1$ and $\PP_2, \leq_2$ are posets, we will denote by $\PP_1 \cdot \PP_2,\leq$ the poset with underlying set $\PP_1 \coprod \PP_2$ and with a partial ordering $\leq$ given by
\begin{eqnarray*}
a \leq b \Leftrightarrow \left\{
\begin{array}{l}
\mbox{$a,b \in \PP_1$ and $a \leq_1 b$} \\
\mbox{$a,b \in \PP_2$ and $a \leq_2 b$} \\
\mbox{$a \in \PP_1$ and $b \in \PP_2$}
\end{array} \right.
\end{eqnarray*}

If $\PP_1,\leq_1$ and $\PP_2,\leq_2$ are posets, then we will write $\PP_1 \stackrel{\rightarrow}{\times} \PP_2,\leq$ for the partially ordered set with underlying set $\PP_1 \times \PP_2$ endowed with the lexicographical ordering.

When writing a poset we will usually suppress the partial ordering, thus writing $\PP$ instead of $\PP,\leq$.  We will always interpret a poset as a small category in the usual way, thus $\PP(a,b)$ consists of a single element if and only if $a \leq b$.  Note that the category $\PP$ is thus not $k$-linear.  We will write $k\PP$ for the $k$-linear additive category generated by $\PP$ (see also Section \ref{subsection:2Colimits}), thus the objects of $k\PP$ are finite direct sums of objects in $\PP$ and $k\PP(a,b) \cong k$ when $a,b \in \PP$ and $a \leq b$.

Let $\CC$ be a Krull--Schmidt category.  By $\ind \CC$ we will denote a set of chosen representatives of isomorphism classes of indecomposable objects of $\CC$.  If $\CC'$ is a Krull--Schmidt subcategory of $\CC$, then we will assume $\ind \CC' \subseteq \ind \CC$.

For $X,Y \in \ind \CC$, we will denote $[X,Y]$ for the full replete (= closed under isomorphisms) additive subcategory of $\CC$ such that for all $Z \in \ind \CC$, we have $Z \in \ind [X,Y] \Leftrightarrow \CC(X,Z) \not= 0$ and $\CC(Z,Y) \not= 0$.  The subcategories $[X,Y[, ]X,Y]$ and $]X,Y[$ are defined in a straightforward manner.

\subsection{Serre duality}
Let $\CC$ be a Hom-finite triangulated $k$-linear category.  A \emph{Serre functor} \cite{BondalKapranov89} on $\CC$ is an additive auto-equivalence $\bS : \CC \to \CC$ such that for every $X,Y \in \Ob \CC$ there are isomorphisms
$$\Hom(X,Y) \cong \Hom(Y,\bS X)^*$$
natural in $X$ and $Y$, and where $(-)^*$ is the vector-space dual.

We will say that $\CC$ has \emph{Serre duality} if $\CC$ admits a Serre functor.  An abelian category $\AA$ is said to satisfy Serre duality if the bounded derived category $\Db \AA$ has Serre duality.

\subsection{Quivers and free categories}\label{subsection:FreeCategories}
A quiver $Q$ is a pair $(Q_0,Q_1)$ consisting of a small set $Q_0$ of vertices and a small set $Q_1$ of arrows between the vertices such that between two elements of $Q_0$, the set of arrows is a small set.  Note that we allow multiple arrows and loops.  A map $f: Q \to Q'$ between quivers consists of a map $f_0: Q_0 \to Q'_0$ of objects and a map $f_1: Q_1 \to Q'_1$ of arrows such that $f_1(A \to B)$ is an arrow from $f_0(A)$ to $f_0(B)$.

We will say a quiver is \emph{finite} if both $Q_0$ and $Q_1$ are finite.  A quiver $Q$ is \emph{locally finite} if every vertex is incident to only a finite number of arrows and $Q$ is said to be \emph{strongly locally finite} if it is locally finite and has no infinite paths of the form $\bullet \rightarrow \bullet \rightarrow \cdots$ or $\cdots \rightarrow \bullet \rightarrow \bullet$.  In particular, a strongly locally finite quiver has no oriented cycles (and hence no loops).

Equivalently, a quiver is strongly locally finite if and only if all indecomposable projective and injective representations have finite length.

A \emph{commutativity condition} on a quiver is a pair of oriented paths, both starting in the same vertex and ending in the same vertex.  We write $\CQ$ for the category whose objects are given by quivers with a set of commutativity conditions and whose morphisms are given by morphisms of quivers $Q \to Q'$ mapping a commutativity condition of $Q$ to a commutativity condition of $Q'$.  Note that the category of quivers is a full subcategory of $\CQ$.

Let $\Cat$ be the 1-category of small categories and functors.  There is an obvious (faithful) forgetful functor $\Cat \to \CQ$ admitting a left adjoint $\Free: \CQ \to \Cat$ (see \cite[Proposition 5.1.6]{Borceux94}).  For a quiver $Q$, the category $\Free Q$ is called the \emph{free category} or the \emph{path category} of $Q$.  Note that the free category of a quiver is not $k$-linear.

\subsection{2-categories and 2-colimits}\label{subsection:2Colimits}
We will give a brief overview of the notation and terminology of 2-categories we will use.  Our main references are \cite{Borceux94, Waschkies04}.  Let $\Cat$ be the 1-category of all small categories, thus the objects are given by small categories and the morphisms by functors.

A \emph{(strict) 2-category}, also called a $\Cat$-category (see \cite{Borceux94b, Kelly05}), is a category enriched over $\Cat$.  Our main example will be $\kVar$, the 2-category of all small $k$-linear additive categories with split idempotents:
\begin{itemize}
\item the 0-cells are given by small $k$-linear additive categories with split idempotents,
\item the 1-cells are the $k$-linear functors,
\item the 2-cells are natural transformations.
\end{itemize}

\begin{definition}
A small $k$-linear additive category with split idempotents will also be called a \emph{$k$-variety} (following \cite{Auslander74, AuslanderReiten74}).  We refer to Section \ref{section:Dualizing} for more information.
\end{definition}

Composition of 1-cells is denoted by $\circ$.  Following \cite{MacLane71, Waschkies04} we will write $\circ$ for vertical composition of 2-cells and $\bullet$ for horizontal composition, as illustrated in Figure \ref{fig:Compositions}.  We have the following equation
\begin{equation}\label{equation:Compatibility}
(\gamma \bullet \delta) \circ (\alpha \bullet \beta)= (\gamma \circ \alpha) \bullet (\delta \circ \beta)
\end{equation}

\begin{figure}
	\centering
		$$\xymatrix@1@C=+70pt{X\ruppertwocell^F{\alpha} \rlowertwocell_H{\beta} \ar[r]|G & Y &
		X\rtwocell^{F_1}_{F_2}{\alpha} &Y \rtwocell^{G_1}_{G_2}{\beta} &Z}$$
  \caption{The left diagram illustrates the vertical composition $\beta \circ \alpha: F \Rightarrow H$ and the right diagram illustrates the horizontal composition $\beta \bullet \alpha: G_1 \circ F_1 \Rightarrow G_2 \circ F_2$.}
  \label{fig:Compositions}
\end{figure}
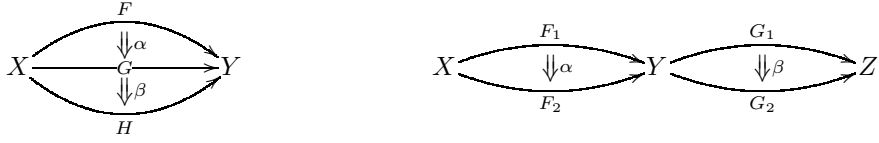

We proceed to define a 2-colimit over a 2-functor.
\begin{definition}\label{definition:TwoFunctor}
Let $\II$ be a small 1-category.  A \emph{2-functor (with strict identity)} $\aa: \II \to \kVar$ is given by the following data:
\begin{enumerate}
\item a 0-cell $\aa(i)$ of $\kVar$ for every $i \in \Ob \II$,
\item a 1-cell $\aa(s): \aa(i) \to \aa(j)$ of $\kVar$ for every morphism $s:i \to j$ in $\II$ and $\aa(1_i) = 1_{\aa(i)}$ for all $i \in \Ob \II$, and
\item a natural equivalence $\Phi(s,t): \aa(t \circ s) \stackrel{\sim}{\rightarrow} \aa(t) \circ \aa(s)$ for all composable morphisms $s,t \in \Ob \II$,
\end{enumerate}
satisfying the following condition: for three composable morphisms $u,t,s \in \Mor \II$, we have the following commutative diagram
$$\xymatrix@C=80pt{
\aa(u \circ t \circ s) \ar[r]^{\Phi(t \circ s, u)} \ar[d]_{\Phi(s, u \circ t)} & \aa(u)\aa(t \circ s) \ar[d]^{1_{\aa(u) \bullet \Phi(s,t)}} \\
\aa(u \circ t)\aa(s) \ar[r]^{\Phi(t,u) \bullet 1_{\aa(s)}} & \aa(u) \aa(t) \aa(s)
}$$
\end{definition}

\begin{example}
For every object $\CC$ of $\kVar$, there is a 2-functor $\CC: \II \to \kVar$ sending every object of $\II$ to $\CC$ and sending every morphism of $\CC$ to the identity on $\CC$.
\end{example}

\begin{definition}
Let $\aa, \bb: \II \to \kVar$ be two 2-functors.  A \emph{2-natural transformation} $\ff: \aa \to \bb$ consists of the following data:
\begin{enumerate}
\item a 1-cell $\ff_i: \aa(i) \to \bb(i)$ of $\kVar$ for every $i \in \Ob \II$, and
\item a natural equivalence $\theta^\ff_s: \bb(s) \circ \ff_i \to \ff_j \circ \aa(s)$ for every morphism $s:i \to j$ in $\II$.
\end{enumerate}
such that for any two composable morphisms $s:i \to j$, $t: j \to k$ in $\II$, we have the following commutative diagram
$$\xymatrix@C=80pt{
\bb(t \circ s) \circ \ff_i \ar[r]^{\Phi^{\bb}(s,t) \bullet 1_{\ff_i}} \ar[dd]_{\theta^\ff_{t \circ s}}& \bb(t)\circ \bb(s) \circ \ff_i \ar[d]^{1_{\bb(t)} \bullet \theta^\ff_s}\\
& \bb(t) \circ \ff_i \circ \aa(s) \ar[d]^{\theta^\ff_t \bullet 1_{\aa(s)}} \\
\ff_k \circ \aa(t \circ s) \ar[r]^{\ff_k \bullet \Phi^{\aa}(s,t)} & \ff_k \circ \aa(t) \circ \aa(s)
}$$
\end{definition}

\begin{definition}
Let $\aa, \bb: \II \to \kVar$ be 2-functors and $\ff, \gg: \aa \to \bb$ be 2-natural transformations.  A \emph{modification} $\Lambda: \ff \to \gg$ consists of a 2-cell $\Lambda_i: \ff_i \to \gg_i$ for all objects $i \in \II$ such that for all $s:i \to j$ in $\II$ the following diagram commutes
$$\xymatrix{
\bb(s) \circ \ff_i \ar[r]^{\theta^\ff_s} \ar[d]_{1_{\bb(s)} \bullet \Lambda_i}&\ff_j \circ \aa(s) \ar[d]^{\Lambda_j \bullet 1_{\aa(s)}} \\
\bb(s) \circ \gg_i \ar[r]^{\theta^\gg_s} & \gg_j \circ \aa(s)}$$
\end{definition}

\begin{definition}
The 2-functors, 2-natural transformations, and modifications form a (strict) 2-category called $2\FF(\II,\kVar)$.
\end{definition}

We can now give the definition of a 2-colimit.

\begin{definition}
Let $\aa: \II \to \kVar$ be a 2-functor.  We say $\aa$ admits a 2-colimit if and only if there exist
\begin{enumerate}
\item a $k$-linear additive category with split idempotents $2\colim \aa$, and
\item a 2-natural transformation $\sigma: \aa \to 2\colim \aa$,
\end{enumerate}
such that for every category $\CC$ the functor
$$(- \circ \sigma): \Hom_{\kVar}(2\colim \aa, \CC) \to \Hom_{2\FF}(\aa, \CC)$$
is an equivalence of categories.
\end{definition}

\begin{remark}\label{remark:Modifications}
The fact that $(- \circ \sigma)$ is fully faithful can be restated as follows:  Let $\ff,\gg: \aa \to \CC$ be 2-natural transformations and let $\Lambda: \ff \to \gg$ be a modification.  Let $f,g: 2\colim \aa \to \CC$ be two functors corresponding to $\ff$ and $\gg$, respectively.  Then there is a unique natural transformation $\alpha: f \to g$ such that the following diagram commutes:
$$\xymatrix{
\ff_i \ar[r]^{\Lambda_i} \ar[d] & \gg_i \ar[d] \\
f \circ \sigma_i \ar[r]_{\alpha \bullet 1_{\sigma_i}}& g \circ \sigma_i}$$
\end{remark}

\begin{remark}
Although the universal property of a 2-colimit is stated for covariant functors, it also holds for contravariant functors.
\end{remark}

We have the following result (\cite[Theorem A.3.4]{Waschkies04}).

\begin{theorem}
Denote by $\CAT$ the 2-category of all small categories.  Let $\II$ be a small 1-category, and $\aa: \II \to \CAT$ a 2-functor.  Then $\aa$ admits a 2-colimit.
\end{theorem}

One may construct this 2-colimit as follows (\cite{Waschkies04}).  We start by constructing a category $\int_\II \aa$.  We set
$$\Ob \int_\II \aa = \{(i,X) \mid i \in \Ob \II, X \in \Ob \aa(i)\} = \coprod_{i \in \II} \Ob \aa(i).$$
and
$$\Hom((i,X),(j,Y)) = \{(s,f) \mid s \in \II(i,j), f \in \Hom_{\aa(j)}(\aa(s)(X),Y) \}.$$
The composition is defined as $(t,g) \circ (s,f) = (t \circ s, g \circ \aa(t)(f))$.  Note that for every $i \in \Ob \II$ there is a natural functor $\aa(i) \to \int_\II \aa$.

We consider the set $\SS$ of morphisms of $\int_\II \aa$ given by
$$\mbox{$\SS = \{(s,f): (i,X) \to (j,Y) \mid f:\aa(s)(X) \to Y$ is an isomorphism$ \}$}.$$

It is then shown that $\left( \int_\II \aa \right)[\SS^{-1}]$ together with the canonical functors $\aa(i) \to \left( \int_\II \aa \right)[\SS^{-1}]$ is a 2-colimit of $\aa$.

We will be interested in 2-colimits in the 2-category $\kVar$ of $k$-varieties.  In general, 2-colimits taken in $\CAT$ of $k$-varieties need not be $k$-varieties.  In this paper, we will follow the outline below.

Let $k\CAT$ be the 2-category of all small $k$-linear categories.  The forgetful 2-functor  $U:k\CAT \to \CAT$ (called a morphism of bicategories in \cite{Benabou67, Ross80}) has a left adjoint $\CAT \to k\CAT$ which associates to every Hom-set a vector space generated by that Hom-set (see \cite[Proposition 6.4.7]{Borceux94b} or \cite[Section 2.5]{Kelly05}).

Likewise, the embedding 2-functor $\kVar \to k\CAT$ from the 2-category of small $k$-linear categories to the 2-category of $k$-varieties also has a left adjoint $k\CAT \to \kVar$ as given in \cite[Proposition 2.3 and Corollary 2.4]{Auslander74}.

Combining these two statements shows that the forgetful 2-functor $F:\kVar \to \CAT$ has a left adjoint $k:\CAT \to \kVar$.  Thus for each small category $\CC$ and each $k$-variety $\DD$, there is an equivalence $\Hom_{\CAT}(\CC, F\DD) \cong \Hom_{\kVar}(k\CC,\DD)$ natural in both components.

In particular, the functor $k:\CAT \to \kVar$ commutes with 2-colimits (see \cite{Ross80}).  Thus if $\aa: \II \to \CAT$ is a 2-functor (as defined in Definition \ref{definition:TwoFunctor}), then $k(2\colim \aa) \cong 2\colim(k \circ \aa)$, where the first 2-colimit is taken in $\CAT$ and the second 2-colimit is taken in $\kVar$.  All 2-colimits we will consider in this article will be of this form.

\begin{example}\label{example:PushOut}
Consider the quivers
$$\xymatrix@R=3pt{
A_2:& x \ar[r] & y \\
A_4:& a \ar[r] & b \ar[r] & c \ar[r] & d \\
&1 \ar[dr] &&& 5 \\ Q: && 3 \ar[r]& 4 \ar[ru] \ar[rd] \\ &2 \ar[ur] &&& 6
}$$

and consider the fully faithful functors given by
\exa
\begin{eqnarray*}
\Free A_2 &\to& \Free A_4 \\
x &\mapsto& a \\
y &\mapsto& d
\end{eqnarray*}
\exb
\begin{eqnarray*}
\Free A_2 &\to& \Free Q \\
x &\mapsto& 3 \\
y &\mapsto& 4.
\end{eqnarray*}
\exc
The 2-colimit in $\CAT$ of the 2-functor $\aa: \II \to \CAT$ whose image is given by
$$\xymatrix{A_2 \ar[r] \ar[d] & Q \\ A_4}$$
can be described as follows.  The category $\int_\II \aa$ is the free category generated by the quiver
$$\xymatrix@R=4pt{
& a \ar[r] & b \ar[r] & c \ar[r] & d \\ \\
& x \ar[uu] \ar[dd] \ar[rrr]&&& y \ar[dd] \ar[uu] \\
1 \ar[dr] &&&&& 5 \\ & 3 \ar[rrr] &&& 4 \ar[ru] \ar[rd] \\ 2 \ar[ur] &&&&& 6}$$
with the extra conditions that the two rectangles commute.  The category $(\int_\II \aa) [\SS^{-1}]$ is given by formally inverting the vertical arrows.  In this way, one obtains a category equivalent to the free category of the quiver $Q'$ given by
$$\xymatrix@R=4pt{\cdot \ar[dr] &&&&& \cdot \\ & \cdot \ar[r] & \cdot \ar[r] & \cdot \ar[r] & \cdot \ar[ru] \ar[rd] \\ \cdot \ar[ur] &&&&& \cdot}$$

When replacing the free categories associated to $A_2, A_4,$ and $Q$ by the categories $kA_2, kA_4,$ and $kQ$, the 2-colimit taken in $\kVar$ is the category $kQ'$.
\end{example}

\section{Preadditive semi-hereditary categories}\label{section:SemiHereditary}

Let $\aa$ be a small preadditive category.  A right $\aa$-module is a contravariant functor from $\aa$ to $\Mod k$, the category of all vector spaces.  The category of all right $\aa$-modules is denoted by $\Mod \aa$.

If $f:\frak{a} \to \frak{b}$ is a functor between small preadditive categories then there is an obvious restriction functor
$$(-)_{\frak{a}}:\Mod(\frak{b})\r \Mod(\frak{a})$$
which sends $N$ to $N \circ f$. This restriction functor has a left adjoint 
$$- \otimes_{\aa} \bb:\Mod(\frak{a})\r \Mod(\frak{b})$$
which is the right exact functor which sends the projective generators
$\frak{a}(-,A)$ in $\Mod(\frak{a})$ to $\frak{b}(-,f(A))$ in $\Mod(\frak{b})$. As usual if $f$ is fully faithful
we have $(N \otimes_{\frak{a}} \frak{b})_\frak{a}=N$.

Let $M$ be in $\Mod(\frak{a})$. We will say that $M$ is \emph{finitely generated} if $M$ is a quotient object of a finitely generated projective.  We say that $M$ is \emph{finitely presented} if $M$ has a presentation
$$P\to Q\to M\to 0$$
where $P,Q$ are finitely generated projectives. It is easy to see that these notions coincides with the ordinary categorical ones.  The full subcategory $\Mod \aa$ spanned by the finitely presented modules will be denoted by $\mod \aa$.  If $\mod \aa$ is an abelian category, we will say that $\aa$ is \emph{coherent}.

Dually we will say that $M$ is \emph{cofinitely generated}\index{cofinitely generated} if it is contained in a cofinitely generated injective. \emph{Cofinitely presented}\index{cofinitely presented} is defined in a similar way.

The categorical interpretation of the latter notions is somewhat less clear.  However if $\frak{a}$ is Hom-finite then both finitely and cofinitely presented representations correspond to each other under duality (exchanging $\frak{a}$ and $\frak{a}^\circ$).

It is well known that $\aa$ is coherent if and only if it has \emph{pseudokernels}, i.e. if for every morphism $A \to B$ in $\aa$ there is a morphism $K \to A$ such that
$$\aa(-,K) \to \aa(-,A) \to \aa(-,B)$$
is exact.  Pseudocokernels are defined in a dual way.

With every object $A$ of $\aa$, we may associate a \emph{standard projective} $\aa(-,A)$ and a \emph{standard injective} $\aa(A,-)^*$.  It is clear that every finitely generated projective is a direct summand of a finite direct sum of standard projectives.  If $\aa$ has finite direct sums and idempotents split in $\aa$, then every finitely generated projective is isomorphic to a standard projective.  Dual notions hold for injective objects.

A morphism $A \to B$ in $\aa$ is a \emph{radical morphism} if there are no maps $X \to A$ and $B \to X$ such that the composition
$$X \to A \to B \to X$$
is an isomorphism, where $X$ is not zero.  The set of all radical morphisms in $\aa(A,B)$ is denoted by $\rad(A,B)$; it is a subspace of $\aa(A,B)$.  With an indecomposable object $A \in \ind \aa$, we may associate in a straightforward way the \emph{standard simple} object $S_A$ as $\aa(-,A) / \rad(-,A)$.

We say that a small preadditive category $\frak{a}$ is \emph{semi-hereditary}\index{category!semi-hereditary} if the finitely presented objects $\mod(\frak{a})$ in $\Mod(\frak{a})$ form an abelian and hereditary category.  The following proposition (\cite{vanRoosmalen06}, see also \cite[Theorem 1.6]{AuslanderReiten75}) makes it easy to recognize semi-hereditary categories in examples.

\begin{proposition}\label{proposition:ShIsLocal}
Let $\frak{a}$ be a small preadditive category, then $\frak{a}$ is semi-hereditary if and only if any full subcategory of $\frak{a}$ with a finite number of objects is semi-hereditary.
\end{proposition}

\begin{lemma}\label{lemma:TensorIsFF}
Let $\frak{b}\to \frak{c}$ be a full embedding of preadditive categories. Then $- \otimes_{\bb} \cc : \mod(\bb) \to \mod(\cc)$ is fully faithful.
\end{lemma}

\begin{proof} This may be checked on objects of the form $\bb(-,B)$ where it is clear. 
\end{proof}

Since $(- \otimes_{\bb} \cc,(-)_{\frak{b}})$ is an adjoint pair, if follows from this lemma that $(- \otimes_{\bb} \cc)_{\frak{b}}$ is naturally isomorphic to the identity functor on $\mod(\frak{b})$.

The following result is a slight generalization of \cite[Proposition 2.2]{AuslanderReiten75}.
  
\begin{proposition}
\label{proposition:SemiHereditaryExact}
Let $\frak{b}\to \frak{c}$ be a full embedding of semi-hereditary categories. 
Then the (fully faithful) functor $-\otimes_{\bb} \cc:\mod(\frak{b})\to \mod(\frak{c})$
is exact.
\end{proposition}

\begin{proof}
Let $M\in \mod(\frak{b})$ and consider a projective resolution $0\r P_1\xrightarrow{\theta} P_0 \r M\r 0$ in $\mod(\frak{b})$. Then $P_1$ is a direct summand of some $\oplus_{i=1}^n\frak{b}(-,B_i)$.  Put $K=\ker(\theta \otimes_{\bb} \cc )$. Then $K$ is a direct summand of $P_1 \otimes_{\bb} \cc$ (since $\frak{c}$ is semi-hereditary) and $K_{\frak{b}}=0$ (since $\theta$ is injective).

Assume that $K$ is non-zero. Since $K$ is a direct summand of $P_1 \otimes_{\bb} \cc$ we obtain a non-zero map $P_1 \otimes_{\bb} \cc \to K$ and hence a non-zero map $\oplus_{i=1}^n\frak{c}(-,B_i)\to K$ and thus ultimately a non-zero element of some $K(B_i)$, contradicting the fact that $K_{\frak{b}}=0$.

Thus $K=0$. If we denote the left satellites of $-\otimes_{\bb} \cc$ by $\Tor_i^{\frak{b}}(\frak{c},-)$ then we have just shown that $\Tor_i^{\frak{b}}(\frak{c},-)=0$ for $i>0$. Hence $-\otimes_{\bb} \cc$ is exact. 
\end{proof}

\section{\texorpdfstring{Dualizing $k$-varieties}{Dualizing k-varieties}}\label{section:Dualizing}

We recall some definitions from \cite{Auslander74, AuslanderReiten74}.  A Hom-finite additive $k$-linear category where idempotents split is called a \emph{finite $k$-variety}.  Such a finite $k$-variety is always Krull--Schmidt.

Denote by $\Modlfd \aa$ the abelian category of locally finite-dimensional right $\aa$-modules, thus the full subcategory of $\Mod \aa$ spanned by all contravariant functors from $\aa$ to $\mod k$.  Note that an additive $k$-linear category where idempotents split is a finite $k$-variety if and only if every standard projective and standard injective lies in $\Modlfd \aa$.

Let $\frak{a}$ be a finite $k$-variety.  There is a duality $D: \Modlfd \aa \to \Modlfd \aa^\circ$ given by sending a module $M \in \Ob \Modlfd \aa$ to the dual $D(M)$ where $D(M)(x) = \Hom_k(M(x), k)$ for all $x \in \frak{a}$.  If this functor induces a duality $D: \mod \frak{a} \to \mod \frak{a}^\circ$ by restricting to the finitely presented objects, then we will say that $\frak{a}$ is a \emph{dualizing $k$-variety}.

The following proposition gives a different characterization of dualizing $k$-varieties (compare with \cite[Theorem 2.4]{AuslanderReiten74}).

\begin{proposition}\label{proposition:Dualizing}
Let $\frak{a}$ be a finite $k$-variety, then $\frak{a}$ is a dualizing $k$-variety if and only if 
\begin{enumerate}
\item $\aa$ has pseudokernels and pseudocokernels, and
\item every standard projective is cofinitely presented and every standard injective is finitely presented.
\end{enumerate}
\end{proposition}

\begin{proof}
Assume that $\aa$ is a dualizing $k$-variety, so that there is a duality $D: \mod \aa \to \mod \aa^\circ$.  Since $\mod \aa^\circ \subseteq \Modlfd \aa^\circ$ is closed under taking cokernels, we know that $\mod \aa \subseteq \Modlfd \aa$ is closed under taking kernels, and hence $\aa$ has pseudokernels.  Likewise, one shows that $\aa^\circ$ has pseudokernels or, dually, that $\aa$ has pseudocokernels.

Let $I$ be a standard injective in $\Modlfd \aa$, so that $DI$ is a standard projective in $\Modlfd \aa^\circ$.  We then know that $DI \in \Ob \mod \aa^\circ$ and thus by duality $I \in \Ob \mod \aa$ so that $I$ is finitely presented.

Let $P$ be a standard projective in $\mod \aa$.  Since $D$ induces a duality between $\mod \aa$ and $\mod \aa^\circ$, we know that $DP$ is an object in $\mod \aa^\circ$, hence finitely presented.  Taking the vector space dual of a projective resolution in $\mod \aa^\circ$ yields an injective resolution in $\mod \aa$.  We obtain that $P$ is cofinitely presented.

For the other direction, let $X \in \mod \aa^\circ$.  Since $X$ is finitely presented, the dual $DX \in \Modlfd \aa$ is cofinitely presented, thus there is an exact sequence
$$0 \to DX \to I \stackrel{f}{\rightarrow} J$$
in $\Modlfd \aa$.  By assumption $\aa$ has pseudokernels such that $\mod \aa$ is an exact subcategory of $\Modlfd \aa$.  It now follows that $DX \in \mod \aa$ so that the functor $D: \Modlfd \aa^\circ \to \Modlfd \aa$ restricts to a functor $D: \mod \aa^\circ \to \mod \aa$.  

Similarly, one shows that there is a functor $D:\mod \aa \to \mod \aa^\circ$ and since $D^2 \cong 1$, these are equivalences.  This shows that $\aa$ is a dualizing $k$-variety.
\end{proof}

\begin{remark}
It follows from Proposition \ref{proposition:Dualizing} that every object in the abelian category $\mod \aa$ has a projective resolution and an injective resolution.  These resolutions need not to be finite.
\end{remark}

\begin{remark}
If $\aa$ is a dualizing $k$-variety, then both $\mod \aa$ and $\mod \aa^\circ$ are abelian.
\end{remark}

\begin{example}
Let $\PP$ be the poset given by $\PP=\{a,b,c\}$ with $a \leq c$ and $b \leq c$, and let $\QQ$ be the poset $\bN \cdot \PP$, thus we may think of $\QQ$ as
$$\xymatrix{&&&&&a \ar[rd] \\
0 \ar[r] & 1 \ar[r] & 2 \ar[r] & \ar@{..}[r]&\ar[ru]\ar[rd]&&c \\
&&&&&b \ar[ru]}$$
Denote $\aa = k\QQ$.  All the standard projective representations of $\aa$ are cofinitely presented and all the standard injective representations are finitely presented, but $\aa$ does not have pseudokernels.  Indeed, the kernel of the obvious map $\aa(-,a \oplus b) \to \aa(-,c)$ is not finitely generated.  We conclude that $\aa$ is not a dualizing $k$-variety.
\end{example}

\begin{observation}\label{observation:MoreDualizing}
\begin{enumerate}
\item
If $Q$ is a strongly locally finite quiver (i.e. all the indecomposable projective and injective representations have finite dimension as $k$-vector space) then the $k$-linear additive path category $kQ$ of $Q$ is a semi-hereditary dualizing $k$-variety.
\item
Let $\aa$ be a dualizing $k$-variety, then $\mod \aa$ is also a dualizing $k$-variety (\cite[Proposition 2.6]{AuslanderReiten74}).
\item
Let $\aa$ be a triangulated finite $k$-variety with Serre duality, then $\aa$ is a dualizing $k$-variety.  The pseudokernels and pseudocokernels are given by the cones, and Serre duality shows that the subcategory of standard injectives and the subcategory of standard projectives coincide.  (Moreover, a triangulated finite $k$-variety is a dualizing $k$-variety if and only if it satisfies Serre duality \cite[Proposition 2.11]{IyamaYoshino08})
\item
It follows from \cite{AuslanderSmalo81} that any functorially finite full subcategory $\bb$ of a dualizing $k$-variety $\aa$ is again dualizing.  This is the case when, for example, the full embedding $\bb \to \aa$ admits a left and a right adjoint (cf. Proposition \ref{proposition:DualizingIfLeftAndRightAdjoint}).
\end{enumerate}
\end{observation}

\begin{example}\label{example:A2}
Let $A_2$ be the quiver $\cdot \rightarrow \cdot$, then the categories $kA_2$, $\mod kA_2$, and $\Db \mod kA_2$ are dualizing $k$-varieties.
\end{example}

\begin{example}
The category $\coh \bP^1$ is not a dualizing $k$-variety since the standard projective module $\Hom(-,\OO_{\bP^1}) \in \mod (\coh \bP^1)$ is not cofinitely generated.  However, $\Db \coh \bP^1$ has Serre duality and is thus a dualizing $k$-variety.
\end{example}

When $\aa$ is a dualizing $k$-variety, then $\mod \aa$ does not necessarily satisfy Serre duality (see Example \ref{example:NoSerreDuality}).  The following theorem gives necessary and sufficient conditions (compare with \cite[Theorem 3.5]{Chen06}).

\theoremA*

\begin{proof}
We denote by $\PP$ and $\II$ the full additive subcategories of $\Mod \aa$ generated by the standard projectives and standard injectives, respectively.

First, assume that $\aa$ is a dualizing $k$-variety and that every object of $\mod \aa$ has finite projective and finite injective dimension.  In this case, the canonical embeddings $\Kb \PP \to \Db \Modlfd \aa$ and $\Kb \II \to \Db \Modlfd \aa$ induce equivalences of $\Kb \PP$ and $\Kb \II$ with $\Db \mod \aa$.  Let $N: \Kb \PP \to \Kb \II$ be the equivalence induced by the Nakayama functor $\PP \to \II$ given by sending $\aa(-,A)$ to $\aa(A,-)^*$.  We define an autoequivalence $\bS: \Db \mod \aa \to \Db \mod \aa$ by
$$\bS:\Db \mod \aa \cong \Kb \PP \stackrel{N}{\rightarrow} \Kb \II \cong \Db \mod \aa.$$
Every $X,Y \in \Db \mod \aa$ correspond to bounded complexes of projectives in $\Kb \PP$ and in order to show $\bS$ is a Serre functor we may reduce to the case where $X \cong \aa(-,A)$ and $Y \cong \aa(-,B)$.  The required isomorphism
$$\Hom(\aa(-,A),\aa(-,B)) \cong \Hom(\aa(-,B), \aa(A,-)^*)^*$$
follows from the Yoneda lemma.

For the other direction, we assume that $\mod \aa$ is abelian and has Serre duality.  For $A,B \in \Ob \mod \aa$, we have $\Ext^i_{\mod \aa}(A,B) \cong \Hom_{\Db \mod \aa}(A,B[i]) \cong \Hom_{\Db \mod \aa}(B[i],\bS A)^*$, where $\bS$ is a Serre functor on $\Db \mod \aa$, such that $\Ext^i(A,-)$ can only be nonzero for finitely many $i$'s.  This shows that every object has finite projective dimension.  Likewise one shows every object has finite injective dimension.

Since every object of $\mod \aa$ has a finite projective resolution, the natural embedding $\Kb \PP \to \Db \mod \aa$ is an equivalence.  There is an equivalence $i: \Kb \II \to \Db \mod \aa$ given by the composition
$$\Kb \II \stackrel{N^{-1}}{\longrightarrow} \Kb \PP \cong \Db \mod \aa \stackrel{\bS}{\longrightarrow} \Db \mod \aa.$$
We want to show that $\aa(A,-)^* \in \Modlfd \aa$ is isomorphic to $i\aa(A,-)^* \in \Db \mod \aa$.  Consider the isomorphisms
\begin{eqnarray*}
\aa(A,-)^* (B) = \aa(A,B)^* &\cong& \Hom_{\Db \mod \aa}(\aa(-,A), \aa(-,B))^* \\
&\cong& \Hom_{\Db \mod \aa}(N^{-1} \aa(A,-)^*, \aa(-,B))^* \\
&\cong& \Hom_{\Db \mod \aa}(\aa(-,B), (\bS \circ N^{-1}) \aa(A,-)^*) \\
&\cong& \Hom_{\Db \mod \aa}(\aa(-,B), i \aa(A,-)^*) \\
&\cong& (i \aa(A,-)^*) (B)
\end{eqnarray*}
where the first and the last isomorphisms are given by the Yoneda Lemma.  Since the above isomorphisms are natural in $A$ and $B$, these isomorphisms give a natural equivalence between $\aa(A,-)^* \in \Modlfd \aa$ and $i\aa(A,-)^* \in \Db \mod \aa$ as required.  Hence $i: \Kb \II \to \Db \mod \aa \subseteq \Db \Modlfd \aa$ is naturally equivalent to the canonical embedding $\Kb \II \to \Db \Modlfd \aa$, implying that every object of $\Db \mod \aa$ has a finite resolution by standard injectives and that every standard injective is finitely presented.

Since $\mod \aa$ is abelian, $\aa$ has pseudokernels.  Since every standard injective is finitely presented, the full subcategory $\II \subseteq \Modlfd \aa$ lies in $\mod \aa$.  Moreover, every $M \in \mod \aa$ is cofinitely presented and hence $\mod (\aa^\circ) \cong (\mod \aa)^\circ$.  Hence $\mod (\aa^\circ)$ is abelian and $\aa$ has pseudocokernels.  Proposition \ref{proposition:Dualizing} yields that $\aa$ is a dualizing $k$-variety.
\end{proof}

\begin{remark}
The main ingredient in the above proof is to lift the Nakayama functor $\PP \to \II$ to obtain an equivalence $N:\Kb \PP \to \Kb \II$.  This is a common way to construct a Serre functor and has been applied in \cite{Chen06, Happel88, MazorchukStroppel08, ReVdB02}, for example.  In particular, one may obtain Theorem \ref{theorem:RepSerreDuality} as a consequence of \cite[Theorem 3.5]{Chen06}.
\end{remark}

\begin{corollary}\label{corollary:ShDualizing} 
Let $\aa$ be a finite $k$-variety.  The category $\mod \aa$ is abelian, hereditary, and has Serre duality if and only if $\aa$ is a semi-hereditary dualizing $k$-variety.
\end{corollary}

\begin{proof}
That $\mod \aa$ is abelian and hereditary if and only if $\aa$ is semi-hereditary is the definition of semi-heredity.

Assume that $\aa$ semi-hereditary dualizing $k$-variety.  In this case, $\mod \aa$ is hereditary and hence the projective and injective dimensions of every object are bounded above by one.  We may then apply Theorem \ref{theorem:RepSerreDuality} to see that $\mod \aa$ has Serre duality.

For the other direction, assume that $\mod \aa$ is abelian, hereditary, and has Serre duality.  In this case, Theorem \ref{theorem:RepSerreDuality} yields that $\aa$ is a dualizing $k$-variety.
\end{proof}

\begin{remark}
The following example shows that there are dualizing $k$-varieties $\aa$ such that $\mod \aa$ has infinite global dimension but $\mod \aa$ still has Serre duality.
\end{remark}

\begin{example}\label{example:InfiniteDimension}
Let $Q$ be an $A_\infty$ quiver with zig-zag orientation where every zig has one more arrow than the preceding zag, labeled as in Figure \ref{fig:ZigZag}.  We define a relation on $Q$ by requiring the composition of any two arrows to be zero.  The associated additive category $\bb$ is a dualizing $k$-variety by Proposition \ref{proposition:Dualizing}.  Since every object of $\mod \bb$ has finite projective dimension, Theorem \ref{theorem:RepSerreDuality} yields that $\mod \bb$ has Serre duality.  However, denoting by $S(v)$ the simple representation associated with the vertex $v \in Q$, we see that the projective dimension of $S(a^i_i)$ is $i$ so that the global dimension of $\mod \bb$ is infinite.
\end{example}

\begin{figure}[tbh]
$$\xymatrix@=10pt{{a^1_0} \ar[rd] && {b^1_0 = a^2_0} \ar[ld] \ar[rd]&&&& {b^2_0 = a^3_0} \ar[dl] \ar[dr] \\
& {a^1_1 = b^1_1} && {a^2_1} \ar[rd] && {b^2_1}\ar[ld]&& {a^3_1} \ar[dr]\\
&&&& {a^2_2 = b^2_2} &&&& {a^3_2} \ar[dr] && \ar@{..}[dl]\\
&&&&&&&&& {a^3_3 = b^3_3}}$$
\caption{}
\label{fig:ZigZag}
\end{figure}

\begin{remark}\label{remark:PiecewiseHereditary}
The category $\mod \bb$ is piecewise hereditary.  Indeed, let $\aa = kQ$ (thus disregarding the conditions on the composition in Example \ref{example:InfiniteDimension}) then one may define a fully faithful functor by
\begin{eqnarray*}
i: \aa^\circ &\to& \Db \mod \bb \\
a^i_j &\mapsto& S(a^i_j)[-j] \\
b^i_j &\mapsto& S(b^i_j)[-j]
\end{eqnarray*}
The essential image of $i$ forms a partial tilting set (\cite{vanRoosmalen06}) in $\Db \mod \bb$ and hence lifts to a fully faithful and exact functor $i:\Db \mod \aa^\circ \to \Db \mod \bb$.  Since $\aa^\circ$ generates $\Db \mod \bb$ as a triangulated category, $i$ will be an equivalence of triangulated categories.
\end{remark}

\begin{example}\label{example:NoSerreDuality}

Let $Q$ be an $A_\infty^\infty$ quiver with linear orientation where the vertices are labeled by integers, as in
$$\xymatrix{\ar@{..}[rd] && -2 \ar[rd] && 0 \ar[rd] && 2 \ar[rd]&& \\
& -3 \ar[ru] && -1 \ar[ru] && 1 \ar[ru] && 3\ar@{..}[ru]}$$
Let $\bb$ be the associated finite $k$-variety one obtains by requiring that the composition of any two arrows is zero.  Then $\bb$ is a dualizing $k$-variety (it is equivalent to the category $\Db \mod kA_2$ from Example \ref{example:A2}).

Since all indecomposables of $\mod \bb$ are either projective--injective or isomorphic to a standard simple, it is easy to check that $\mod \bb$ has left and right almost split sequences, but Theorem \ref{theorem:RepSerreDuality} shows that $\Db \mod \bb$ has no Serre functor since every standard simple module of $\bb$ has infinite projective dimension.
\end{example}

\begin{remark}
As in Remark \ref{remark:PiecewiseHereditary}, one shows that the category $\mod \bb$ from the previous example is piecewise hereditary.  Denote by $\aa$ the k-linear path category $kQ$ (without relations).  There is an equivalence $F: \Db \mod \aa^\circ \to \Db \mod \bb$ induced by
\begin{eqnarray*}
i: \aa^\circ &\to& \Db \mod \bb \\
n &\mapsto& S(n)[-n]
\end{eqnarray*}
where $S(n)$ is the standard simple associated with the indecomposable $n \in \Ob \bb$.  Note that $\mod \aa^\circ$ has nonzero projectives but no nonzero injectives, hence it does not satisfy Serre duality (see \cite[Theorem A]{ReVdB02}).
\end{remark}

\begin{remark}
Since the category of projectives $\PP$ and the category of injectives $\II$ in $\mod \bb$ coincide, we may lift the Nakayama functor $N: \PP \to \II$ to a Serre functor on the category of perfect complexes (see for example \cite{Chen06,MazorchukStroppel08}).

The category of perfect complexes corresponds under the functor $F: \Db \mod \aa^\circ \to \Db \mod \bb$ to the subcategory $\Db \modfd \aa^\circ$ of $\Db \mod \aa^\circ$, where $\modfd \aa^\circ$ is the category of finite dimensional modules of $\aa^\circ$.  It has been shown in \cite{Ringel02} that $\modfd \aa^\circ$ does indeed satisfy Serre duality.
\end{remark}

Let $\aa$ be a finite $k$-variety and $A$ an indecomposable object in $\aa$.  A map $f:A \to M$ is called \emph{left almost split} if every nonsplit map $A \to B$ factors through $f$.  Dually, a map $g: N \to A$ is called \emph{right almost split} if every nonsplit map $B \to A$ factors through $g$.

We will say a finite $k$-variety $\aa$ is \emph{locally finite and locally discrete} if every indecomposable object $A$ of $\aa$ admits a left almost split map $A \to M$ and a right almost split map $N \to A$.

\begin{example}
Let $Q$ be a finite quiver.  The path category $kQ$ is locally discrete and locally finite.
\end{example}

\begin{example}
Let $\PP$ be the poset $\bN \cdot \{+ \infty\}$.  We may draw the Auslander--Reiten quiver of $k\PP$ as
$$\xymatrix{
0 \ar[r] & 1 \ar[r] & 2 \ar[r] & 3 \ar@{..}[r]& {+\infty}
}$$
It is clear that $k\PP$ is not locally discrete since there is no right almost split map $N \to (+\infty)$.  Note that $+\infty$ is an accumulation point of $\PP$.
\end{example}

We now give an equivalent formulation of these properties.

\begin{proposition}\label{proposition:LocallyFinite}
A finite $k$-variety is locally finite and locally discrete if and only if all standard simples of $\aa$ are finitely presented and cofinitely presented.  Furthermore, every dualizing $k$-variety is locally finite and locally discrete.
\end{proposition}

\begin{proof}
Assume that $\aa$ is locally finite and locally discrete.  For an indecomposable $A \in \ind \aa$, let $N \to A$ be a right almost split map which gives rise to a map in $\mod \aa$
$$\aa(-,N) \longrightarrow \aa(-,A).$$
of which the cokernel is the standard simple $S_A$.  Dually, one shows all standard simples are cofinitely presented.

Next, assume that all standard simples are finitely and cofinitely presented.  We prove that every indecomposable $A \in \ind \aa$ admits a left almost split map $N \to A$, for a certain object $N \in \Ob \aa$.  Consider a presentation of $S_A$
$$Q \longrightarrow \aa(-,A) \longrightarrow S_A \longrightarrow 0.$$
We may write the projective $Q$ as $\aa(-,N)$ and, since $S_A \cong \aa(-,A)/\rad \aa(-,A)$, the induced map $N \to A$ is right almost split.  Dually, one proves a $A$ admits a left almost split map $A \to M$.

For the last part, let $\aa$ be a dualizing $k$-variety and let $A \in \Ob\aa$ be an indecomposable object. Denote by $S_A \in \Ob \Modlfd \aa$ the associated standard simple representation.  We know there is an epimorphism $\aa(-,A) \to S_A$ in $\Modlfd \aa$ and, since $DS_A$ is a simple object in $\Modlfd \aa$, there is an epimorphism $\aa(A,-) \to DS_A$ in $\Modlfd \aa^\circ$.  Applying $D$ yields that $S_A$ is the image of a map $f:\aa(-,A) \to \aa(A,-)^*$.  Proposition \ref{proposition:Dualizing} yields that $\Ker f \in \mod \aa$.

There is now a short exact sequence $0 \to \ker f \to \aa(-,A) \to S_A \to 0$ which shows $S_A \in \Ob \mod \aa$.
\end{proof}

\begin{remark}
It has already been shown in \cite[Proposition 3.2]{AuslanderReiten74} that all standard simple representations of a dualizing $k$-variety are finitely presented.
\end{remark}

\section{(Co)reflective subvarieties}

Recall that a full replete (= closed under isomorphisms) subcategory of a category is called \emph{reflective} or \emph{coreflective} if the embedding has a left or a right adjoint, respectively.

In this section we will consider such reflective and coreflective subcategories of finite $k$-varieties.  These subcategories occur often in the semi-hereditary case (Proposition \ref{proposition:WhenAdjoints}) and the following proposition show they preserve the properties we are interested in.

\begin{proposition}\label{proposition:DualizingIfLeftAndRightAdjoint}
Let $\aa$ be a dualizing $k$-variety, $\bb$ a finite $k$-variety, and $i:\bb \to \aa$ a full embedding admitting a left adjoint $i_L$ and a right adjoint $i_R$.  Then $\bb$ is also a dualizing $k$-variety.  Furthermore, if (the abelian category) $\mod \aa$ satisfies Serre duality, then so does $\mod \bb$.
\end{proposition}

\begin{proof}
Let $B_1 \to B_2$ be a map in $\bb$.  Let $K,C \in \Ob \aa$ be a pseudokernel and pseudocokernel of the map $iB_1 \to iB_2$ in $\aa$, respectively.  Then $i_R K$ is a pseudokernel and $i_L C$ is a pseudocokernel of $B_1 \to B_2$.

Let $B \in \Ob \bb$ and let $\aa(-,A_1) \to \aa(-,A_0) \to \aa(iB,-)^* \to 0$ be a projective resolution in $\mod \aa$, then $\bb(-,i_R A_1) \to \bb(-,i_R A_0) \to \bb(B,-)^* \to 0$ is a projective resolution in $\mod \bb$.  This shows that every standard injective in $\mod \bb$ is finitely presented.  Likewise, one shows that every (standard) projective in $\bb$ is cofinitely presented.  Proposition \ref{proposition:Dualizing} yields that $\bb$ is a dualizing $k$-variety.

Assume now furthermore that $\mod \aa$ satisfies Serre duality.  By Theorem \ref{theorem:RepSerreDuality}, it suffices to show that every object of $\mod \bb$ has finite projective and finite injective dimension.  Thus consider an $M \in \mod \bb$ and assume a projective resolution of $M \otimes_\bb \aa$ is given by
$$0 \to \aa(-,A_n) \to \cdots \to \aa(-,A_1) \to \aa(-,A_0) \to M \otimes_\bb \aa \to 0.$$
Restricting this to $\bb$ and using that $(- \otimes_\bb \aa)_\bb$ is equivalent to the identity functor on $\mod \bb$, we find
$$0 \to \bb(-,i_R A_n) \to \cdots \to \bb(-,i_R A_1) \to \bb(-,i_R A_0) \to M \to 0$$
such that $M$ has finite projective dimension.  Analogously, one shows that every object in $\mod \bb$ has a finite resolution by standard injectives.
\end{proof}

The following proposition gives some examples of reflective and/or coreflective subvarieties. 

\begin{proposition}\label{proposition:WhenAdjoints}
Let $\aa$ be a semi-hereditary finite $k$-variety.  
\begin{enumerate}
\item Let $\MM$ be a set of objects of $\mod \aa$ such that $\sum_{M \in \MM} \dim M(A) < \infty$ for all $A \in \Ob \aa$ and $\Ext^1(M_1,M_2) = 0$ for all $M_1,M_2 \in \MM$.  Define a full subcategory $\aa_{^\perp \MM}$ of $\aa$ given by 
$$B \in \aa_{{}^\perp \MM} \Longleftrightarrow \forall M \in \MM: M(B) = 0.$$
The embedding $i: \aa_{^\perp \MM} \to \aa$ has a left and right adjoint.
\item Let $Y \in \ind \aa$, then the embedding $\supp \aa (-,Y) \to \aa$ has a left adjoint.
\item Let $X,Y \in \ind \aa$.  If $\aa$ is a dualizing $k$-variety, then the embedding $[X,Y] \to \aa$ has a left and a right adjoint.
\item Let $X,Y \in \ind \aa$ and let $\bb$ be the full subcategory of $\aa$ such that no object in $\bb$ has a direct summand in $[X,Y]$.  Then the embedding $\bb \to \aa$ has a left and a right adjoint.
\end{enumerate}
\end{proposition}

\begin{proof}
\begin{enumerate}
\item
We start by defining the right adjoint $i_R: \aa \to \aa_{{}^\perp \MM}$.  Let $A_1 \to A_2$ be a map in $\aa$ and consider the following commutative diagram with exact rows
$$\xymatrix{
0 \ar[r] & \aa(-,A_1') \ar[r]\ar@{-->}[d] & \aa(-,A_1) \ar[r]\ar[d] & \bigoplus_{M \in \MM} \Hom(\aa(-,A_1), M)^* \otimes_k M \ar[d] \\
0 \ar[r] & \aa(-,A_2') \ar[r] & \aa(-,A_2) \ar[r] & \bigoplus_{M \in \MM} \Hom(\aa(-,A_2), M)^* \otimes_k M
}$$
Note that both $\bigoplus_{M \in \MM} \Hom(\aa(-,A_1), M)^* \otimes_k M$ and $\bigoplus_{M \in \MM} \Hom(\aa(-,A_2), M)^* \otimes_k M$ lie in $\mod \aa$ (since $\sum_{M \in \MM} \dim M(A) < \infty$ for all $A \in \Ob \aa$) and that $M(A'_1) = M(A'_2) = 0$ for all $M \in \MM$ (since $\Ext^1(M_1,M_2) = 0$ for all $M_1,M_2 \in \MM$) so that $A'_1, A'_2 \in \Ob \aa_{{}^\perp \MM}$.

Since the kernel, the Yoneda embedding, and the tensor product are functorial, it is easy to see that the correspondence from $A_1 \to A_2$ to $A_1' \to A_2'$ is functorial.  It is readily checked that this functor is right adjoint to the embedding.

The left adjoint is defined in a dual way.

\item For every object $A \in \Ob \aa$, consider the canonical map $f_A: \aa(-,A) \to \Hom(A,Y)^* \otimes \aa(-,Y)$.  Since $\mod \aa$ is hereditary, the image $\im f_A$ is a representable functor.  Choose an object $A' \in \Ob \aa$ such that $\im f_A \cong \aa(-,A')$.  Note that $A'$ is a maximal direct summand of $A$ lying in $\supp \aa(-,Y)$.

Let $A \to B$ be a map in $\aa$.  The following commutative diagram
$$\xymatrix{
\aa(-,A) \ar[rr]^-f \ar[rd] \ar[ddd]&& \Hom(A,Y)^* \otimes \aa(-,Y) \ar[ddd] \\
&\aa(-,A') \ar[ru] \ar@{-->}[d] \\
&\aa(-,B') \ar[rd] \\
\aa(-,B) \ar[rr]^-g \ar[ru] && \Hom(B,Y)^* \otimes \aa(-,Y)
}$$
gives a map $\aa(-,A') \to \aa(-,B')$.  The objects $A'$ and $B'$ lie in $\supp \aa(-,Y)$ and the correspondence $A \to B$ to $A' \to B'$ is functorial.  It is readily checked that this defines a left adjoint $i_L: \aa \to \supp \aa(-,Y)$ to the functor $i$.

\item We will only prove that the embedding has a left adjoint; the right adjoint is dual.  Let $\aa(-,Z)$ be a kernel of the canonical map $\aa(-,Y) \to \aa(X,Y) \otimes \aa(X,-)^*$.  We show that for any $A \in \supp \aa(-,Y)$ we have $A \in \supp \aa(X,-)^*$ if and only if $\aa(A,Z) = 0$.  Note that by applying the exact functor $\Hom(-,\aa(X,-)^*)$ to the exact sequence
$$0 \to \aa(-,Z) \to \aa(-,Y) \to  \aa(X,Y) \otimes \aa(X,-)^*$$
we see that $\Hom(\aa(-,Z),\aa(-,X)^*) \cong \aa(X,Z)^* = 0$.  Since $\aa$ is semi-hereditary, $\aa(A,Z) \not= 0$ implies $A \not\in \supp \aa(X,-)^*$.

For the other direction: if $\aa(A,Z) = 0$, then evaluating the exact sequence $0 \to \aa(-,Z) \to \aa(-,Y) \to \aa(X,Y) \otimes \aa(X,-)^*$ in $A$ shows that $A \in \supp \aa(X,-)^*$.

By applying the first part of this proposition with $\MM = \{ \aa(-,Z) \}$, we know that $[X,Y] \to \supp \aa(-,Y)$ has a left adjoint.  The second part of the proposition implies $\supp \aa(-,Y) \to \aa$ has a left adjoint.

\item Consider the canonical map $f: \aa(X,Y)^* \otimes \aa(-,Y) \to \aa(X,-)^*$.  Explicitly, there is a map
\begin{eqnarray*}
f_Z: \aa(X,Y)^* \otimes \aa(Z,Y) &\to& \aa(X,Z)^* \\
\varphi \otimes g &\mapsto& \varphi(g \circ -)
\end{eqnarray*}
for each $Z \in \aa$.  Let $M \cong \im f$; since $M$ is the image of a map in $\mod \aa$, we know that $M$ also lies in $\mod \aa$.  Moreover, it follows from the above description that $M(Z) = 0$ if and only if there are no maps in $\aa(X,Y)$ which factor through $Z$, thus $M(Z) = 0$ if and only if $Z \in \bb$.

Next, we show that $\Ext(M,M) = 0$.  Let $0 \to M \to \aa(X,-)^* \to J \to 0$ be an injective resolution of $M$.  However, $\Ext(M,M)$ is a quotient group of $\Hom(M,J)$ which is a subgroup of $\Hom(\aa(X,Y)^* \otimes \aa(-,Y),J)$.  Using the universal property of $f$, we can show that every map $\aa(-,Y) \to \aa(X,-)^*$ factors through $f$ and hence through $Z$.  This shows that $\Hom(\aa(-,Y),J) = 0$ and thus also that $\Ext(M,M) = 0$.

We can now apply the first part of the proposition to obtain the required property.
\end{enumerate}
\end{proof}

The following example shows that the conditions in the first statement of Proposition \ref{proposition:WhenAdjoints} are necessary.

\begin{example}\label{example:WhenAdjoints}
Let $\LL = \bN \cdot - \bN$.  By choosing $\MM = \{M_i\}_{i \in \bN}$ where either
\begin{enumerate}
\item $M_i = k\LL(-,i)$, or
\item $M_i = S_i$, where $S_i$ is the simple representation associated with $i \in \bN \subset \Ob k\LL$
\end{enumerate}
we find an associated subcategory $\bb = k\LL_{{}^\perp \MM}$ such that the natural embedding $\bb \to k\LL$ does not have a left adjoint.  Here the subcategory $\bb$ is generated by $-\bN$.
\end{example}

Propositions \ref{proposition:DualizingIfLeftAndRightAdjoint} and \ref{proposition:WhenAdjoints} will be used in the next section to deduce properties of semi-hereditary dualizing $k$-varieties.  We will give two examples to illustrate their use; a rigorous treatment follows in Proposition \ref{proposition:ThreadsInVarieties}.

\begin{example}\label{example:NotLocallyDiscrete}
Let $\PP$ be the poset given by the following diagram
$$\xymatrix@R=3pt{
&&&&\ar[r]&\ar[r]&\ar[r]& \\
\ar[r]&X\ar[r]&{X_1}\ar[r]&\ar@{.}[ru]\ar@{.}[rd] \\
&&&&\ar[r]&{Y_1}\ar[r]&Y\ar[r]&
}$$
and let $\aa$ be the semi-hereditary finite $k$-variety $k\PP$.  By removing $]X,Y[ = [X_1,Y_1]$ as in Proposition \ref{proposition:WhenAdjoints}, we get a semi-hereditary finite $k$-variety $\bb$ which is of the form $k\PP'$ where $\PP'$ is given by 
$$\xymatrix@R=3pt{
&&&&\ar[r]&\ar[r]&\ar[r]& \\
\ar[r]&\ar[r]&X\ar@{.}[rru]\ar[rrrd] \\
&&&&&Y\ar[r]&\ar[r]&
}$$
It is clear that $\bb$ is not locally finite and locally discrete as there is no left almost split map $X \to M$.  We conclude that, although $\aa$ is a finite $k$-variety, it is not a dualizing $k$-variety.
\end{example}

\begin{example}\label{example:NotLocallyFinite}
Let $\PP$ be the semi-hereditary finite $k$-variety given by
$$\xymatrix@R=5pt{X\ar[r]\ar[dd]&{X_1}\ar[r]\ar[dd]&\ar[r]\ar[dd]&\ar@{.}[r]&\ar[r]&\ar[r]&{Y_1}\ar[r]&Y \\ &&&\cdots \\
&&&}$$
and let $\aa$ be $k\PP$.  Again we remove $]X,Y[ = [X_1,Y_1]$ as in the Proposition \ref{proposition:WhenAdjoints} and obtain a semi-hereditary finite $k$-variety $\bb$ of the form $k\PP'$ where $\PP'$ is given by
$$\xymatrix@R=5pt{X\ar[dd]\ar[rrrrr]\ar[ddr]\ar[ddrr]\ar[ddrrr]&&&&&Y \\ &&&&\cdots \\
&&&&}$$
Here $\bb$ is not locally finite and locally discrete.  Indeed there is no left almost split map $X \to M$.  Again we conclude $\aa$ is not a dualizing $k$-variety.
\end{example}

\section{Threads}

Let $\aa$ be a semi-hereditary dualizing $k$-variety.  Recall from Proposition \ref{proposition:LocallyFinite} that every indecomposable object $A \in \Ob \aa$ admits a left almost split map $A \to M$ and a right almost split map $N \to A$.  Examples \ref{example:NotLocallyDiscrete} and \ref{example:NotLocallyFinite} suggest that most indecomposables lying in a convex subcategory should be---in some sense---nicely behaved.  To make this rigorous, we introduce threads and thread objects (see also \cite[Definition 6.7]{BergVanRoosmalen10}).  The discussion in this section is similar to the discussion in \cite[Section 6]{BergVanRoosmalen10}.

\begin{definition}
If $\aa$ is a semi-hereditary dualizing $k$-variety, an indecomposable object $X \in \ind \aa$ will be called a \emph{thread object} if $X$ has a unique direct predecessor and a unique direct successor, or equivalently, there is a left almost split map $X \to M$ and a right almost split map $N \to X$ where $M$ and $N$ are indecomposable.  In this case, we will denote the representatives of $M$ and $N$ in $\ind \aa$ by $X^+$ and $X^-$, respectively.

For $X,Y \in \ind \aa$, the subcategory $[X,Y]$ will be called a \emph{thread} if every indecomposable object in $[X,Y]$ is a thread object in $\aa$.  A thread is called \emph{maximal} if it is not a proper subset of another thread.  It is called \emph{infinite} if it contains infinitely many nonisomorphic indecomposables.

An indecomposable object which is not a thread object is called a \emph{nonthread object}.
\end{definition}

\begin{proposition}\label{proposition:ThreadsInVarieties}
Let $\aa$ be a semi-hereditary dualizing $k$-variety, and $X,Y,Z \in \ind \aa$, then
\begin{enumerate}
\item $\supp \aa(X,-)^*$ and $\supp \aa(-,X)$ have only finitely many nonthread objects,
\item if $[X,Y]$ is a thread, then $\dim \aa(X,Y) = 1$,
\item if $[X,Y]$ and $[X,Z]$ are threads, then either $[X,Y] \subseteq [X,Z]$ or $[X,Z] \subseteq [X,Y]$.
\end{enumerate}
\end{proposition}

\begin{proof}
\begin{enumerate}
\item It follows from Proposition \ref{proposition:Dualizing} that $\aa(-,X)$ is cofinitely presented; let
$$0 \to \aa(-,X) \to \aa(I,-)^* \to \aa(J,-)^* \to 0$$
be an injective resolution in $\mod \aa$.  Assuming that $\supp \aa(-,X)$ has finitely many nonthread objects, one can use that $\aa$ is semi-hereditary to show that either $I$ or $J$ has infinitely many direct summands.

Alternatively, note that for every nonthread object $Y \in \ind \supp \aa(-,X)$, there is a indecomposable direct summand $I_i$ of $I$ such that $Y \in [I_i,X]$.  One can embed $\aa$ into $\Db \mod \aa$ and apply \cite[Lemma 6.17(1)]{BergVanRoosmalen10} to see that $[I_i,X]$ has only finitely many nonthread objects.

\item Write $V = \aa(X,Y)^*$ and consider the canonical map $\aa(-,Y) \otimes V \to \aa(X,-)^*$ with image $F \in \mod \aa$.  We have the following short exact sequences
\begin{equation}\label{equation:ThreadSequence1}
\xymatrix@1{0 \ar[r] & \aa(-,A) \ar[r] & \aa(-,Y) \otimes V \ar[r] & F \ar[r] & 0}
\end{equation}
\begin{equation}\label{equation:ThreadSequence2}
\xymatrix@1{0 \ar[r] & F \ar[r] & \aa(X,-)^* \ar[r] & \aa(B,-)^* \ar[r] & 0}
\end{equation}
It is clear that every direct summand of $A$ maps nonzero to $Y$.  Furthermore, from the second short exact sequence we obtain $\dim F(X) = 1$ so that $\dim V > 1$ would imply $\dim \aa(X,A) \not= 0$.  Thus at least one indecomposable direct summand $A_1$ of $A$ lies in $[X,Y[$, hence $A_1$ is a thread object with direct successor $A_1^+ \in \ind ]X,Y] \subset \ind \aa$.  It follows from the first short exact sequence that $\dim F(A_1) < \dim F(A_1^+)$.

However, applying $\Hom(\aa(-,Y),-)$ to the sequence (\ref{equation:ThreadSequence2}) yields $\aa(B,Y) = 0$.  In particular, since $\aa$ is semi-hereditary, we have for every $Z \in [X,Y]$ that $\dim \aa(B,Z) = 0$ and thus $\dim F(Z) = \dim \aa(X,Z)$.  We find $\dim \aa(X,A_1) < \dim \aa(X,A_1^+)$, a contradiction since $\aa$ is semi-hereditary.

\item Let $i: [X,Y] \to \aa$ be the natural embedding.  By Proposition \ref{proposition:WhenAdjoints} we know that the embedding has a right adjoint $i_R: \aa \to [X,Y]$.  Let $A$ be an indecomposable direct summand of $i_R(Z)$.

If $A \cong Y$, then $\aa(Y,Z) \not= 0$ so that $[X,Y] \subseteq [X,Z]$.

Therefore, we will assume that $A \not \cong Y$ and show that $iA \cong Z$.  In this case, we have $A^+ \in \ind [X,Y]$.  Since $A \to A^+$ is left and right almost split, we know that $\dim \aa(iA,Z) = \dim \aa(iA^+,Z)$ if $Z \not\cong iA$.  Using adjointness, we find $\dim\Hom(A,i_R Z) = \dim\Hom(A^+,i_R Z)$.  This last statement is impossible since $A$ is a direct summand of $i_R Z$.  We find that $iA \cong Z$ and thus $\aa(Z,Y) \not= 0$.  We conclude that $[X,Z] \subseteq [X,Y]$
\end{enumerate}
\end{proof}

\begin{corollary}\label{corollary:ThreadsPosets}
Let $\aa$ be a semi-hereditary dualizing $k$-variety. Every thread $[X,Y]$ is equivalent to $k \LL$ for a linearly ordered locally discrete poset $\LL$ with a maximal and a minimal element.
\end{corollary}

\begin{proof}
Write $\bb = [X,Y]$ and let $\LL = \ind [X,Y]$.  We define a poset structure on $\LL$ by
$$A \leq B \Longleftrightarrow \aa(A,B) \not= 0.$$
It follows from Proposition \ref{proposition:ThreadsInVarieties} that $\LL,\leq$ is a linearly ordered locally discrete poset with unique minimal and maximal element.  We will define a functor $k \LL \to \bb$ as follows.  For every $A \in \bb$, choose a nonzero element $b_{X,A} \in \aa (X,A)$.

For every $A,B \in \ind \bb$ with $\bb(A,B) \not= 0$, this gives a unique element $b_{A,B} \in \bb(A,B)$ with $b_{X,B} = b_{X,A} \circ b_{A,B}$.

The functor $k\LL \to \bb$, which is the identity on objects and sending the element $(A \leq B) \in k\LL(A,B)$ to $b_{A,B} \in \bb(A,B)$, is fully faithful and essentially surjective, and hence an equivalence.
\end{proof}

\begin{corollary}\label{corollary:MapsOutOfThreads}
Let $\aa$ be a semi-hereditary dualizing $k$-variety, and let $[X,Y]$ be an thread.  If $Z \in \ind \aa$ does not lie on $[X,Y]$, then every map $f:X \to Z$ factors through a nonzero map $X \to Y^+$.
\end{corollary}

\begin{proof}
Let $\bb$ be the full subcategory of $\aa$ generated by objects not supported on $[X,Y]$.  By Proposition \ref{proposition:WhenAdjoints}, the embedding $i: \bb \to \aa$ has a left adjoint $i_L:\aa \to \bb$.

Using Proposition \ref{proposition:ThreadsInVarieties}, it is straightforward to check that $i_L X \cong Y^+$ and $\dim \Hom(X,Y^+) = 1$.

For a $Z \in \bb$, every morphism in $\Hom(X,Z)$ factors through $\eta(1_X) \in \Hom(i\circ i_L X,X)$ where $\eta: 1 \to i\circ i_L$ is the unit of the adjunction $(i_L,i)$.
\end{proof}

In the next proof, we will use the notions of sinks and sources.  An indecomposable object $S \in \ind \aa$ is called a \emph{sink} if $S \to 0$ is a left almost split morphism.  Dually, $S \in \ind \aa$ is called a \emph{source} if $0 \to S$ is a right almost split morphism.  Note that sinks and sources are never thread objects.

\begin{lemma}\label{lemma:SinksAndSources}
For any nonzero $Y \in \aa$, the subcategory $\supp \aa(-,Y)$ contains a source and the subcategory $\supp \aa(Y,-)^*$ contains a sink.
\end{lemma}

\begin{proof}
We will only show that $\supp \aa(-,Y)$ contains a source; the other statement is dual.  Proposition \ref{proposition:Dualizing} yields that $\aa(-,Y)$ is cofinitely presented.  Let $\aa(-,Y) \to \aa(I,-)^*$ be an injective envelope.  We show that any indecomposable direct summand $S$ of $I$ is a source.  Let $M \to S$ be a minimal right almost split morphism.  Seeking a contradiction, assume that $M \not= 0$.  Since $\aa$ is semi-hereditary, we know (by Proposition \ref{proposition:ShIsLocal}) that $M \to S$ induces a monomorphism $\aa(S,-) \to \aa(M,-)$ or thus an epimorphism $\aa(M,-)^* \to \aa(S,-)^*$.

Let $I'$ be the object $I$ after replacing $S$ by $M$.  There is an obvious epimorphism $\aa(I',-)^* \to \aa(I,-)^*$ and using the projectivity of $\aa(-,Y)$, we find the following commutative diagram
$$\xymatrix{\aa(-,Y) \ar@{=}[d] \ar[r] & \aa(I',-)^* \ar[d] \\ \aa(-,Y) \ar[r] & \aa(I,-)^*}$$
From the minimality of $\aa(-,Y) \to \aa(I,-)^*$ follows that the epimorphism $\aa(I',-)^* \to \aa(I,-)^*$ splits and hence that $M \to S$ splits.  A contradiction.  We conclude that $S \in \ind \aa$ is a source.
\end{proof}

\begin{corollary}
Every thread in a semi-hereditary dualizing $k$-variety lies inside a maximal thread.
\end{corollary}

\begin{proof}
Let $[X,Y]$ be a thread in a semi-hereditary dualizing $k$-variety $\aa$.  By Proposition \ref{proposition:ThreadsInVarieties} there are only finitely many nonthread objects in $\supp \aa(-,Y)$ and by Lemma \ref{lemma:SinksAndSources} we know that there is at least one nonthread object in $\supp \aa(-,Y)$.  Proposition \ref{proposition:ShIsLocal} implies that there are no cycles in $\aa$, so that there is a nonthread object $Z \in \supp \aa(-,Y)$ such that $Z$ is the only nonthread object in $[Z,Y]$.  Let $Z \to M$ be a minimal left almost split morphism.  Since $Z \in \supp \aa(-,Y)$, we know that there is at least one direct summand $X'$ of $M$ which lies in $[Z,Y]$. Since $[X',Y] \subset [Z,Y]$, we know that $[X',Y]$ is a thread, and it follows from the dual of Proposition \ref{proposition:ThreadsInVarieties} that $[X,Y] \subseteq [X',Y]$.

Dually, one finds a thread object $Y'$ such that $[X',Y']$ is a thread containing $[X',Y]$ and such that ${Y'}^+$ is a nonthread object.  The thread $[X',Y']$ is the required maximal thread.
\end{proof}

\begin{proposition}\label{proposition:AllThreads}
Let $\aa$ be a semi-hereditary dualizing $k$-variety and let $\{[X_j,Y_j]\}_{j \in J}$ be a set of threads such that $X_j^- \not\in [X_l,Y_l]$ and $Y_j^+ \not\in [X_l,Y_l]$, for all $j,l \in J$.  Let $\bb$ be the full subcategory of $\aa$ consisting of all objects without direct summands in one of these threads.  Then $\bb$ is a semi-hereditary dualizing $k$-variety and the embedding $i:\bb \to \aa$ has both a left and a right adjoint.
\end{proposition}

\begin{proof}
By Propositions \ref{proposition:ShIsLocal} and \ref{proposition:DualizingIfLeftAndRightAdjoint} it suffices to show $i$ has both a left and a right adjoint.  We will show that $i$ has a right adjoint, the proof for the left adjoint is similar.

To show that $i$ has a right adjoint, we will show that, for each $A \in \ind \aa$, the functor $\aa(i-,A)|_\bb: \bb \to \mod k$ is representable.  If $A \in \bb$, then $\aa(i-,A)|_\bb \cong \bb(-,A)$.

We may thus assume that $A \not\in \bb$ or, equivalently, that there is a $j \in J$ such that $A \in [X_j,Y_j]$.  Let $\bb_j$ be the full subcategory of $\aa$ consisting of objects with no nonzero direct summands lying in $[X_j,Y_j]$.  It follows from Proposition \ref{proposition:WhenAdjoints} that $\bb_j \to \aa$ has a right adjoint $R_j: \aa \to \bb_j$.  Since $[X_j,A]$ is a thread, the dual of Corollary \ref{corollary:MapsOutOfThreads} shows that $R_j(A) \cong X^-_j$.  Since $X_l^- \not\in [X_j,Y_j]$ for all $l \in \JJ$, we know that $X^-_j \in \bb$.  It is now easy to see that $\aa(i-,A) \cong \bb(-,X^-_j)$.

We conclude that $\bb \to \aa$ has a right adjoint.
\end{proof}

\begin{corollary}\label{corollary:AllThreads}
Let $\aa$ be a semi-hereditary dualizing $k$-variety and let $\{[X_j,Y_j]\}_{j \in J}$ be a set of maximal threads.  Let $\bb$ be the full subcategory of $\aa$ consisting of all objects without direct summands in one of these threads.  Then $\bb$ is a semi-hereditary dualizing $k$-variety and the embedding $i:\bb \to \aa$ has both a left and a right adjoint.
\end{corollary}

\begin{proof}
If $[X_j,Y_j]$ is a maximal thread, then neither $X_j^-$ nor $Y_j^+$ are thread objects.  Hence the conditions in Proposition \ref{proposition:AllThreads} are satisfied and the conclusion follows.
\end{proof}

\section{Classification by thread quivers}\label{subsection:DualizingThreadQuiver}

In this section, we will classify semi-hereditary dualizing $k$-varieties by means of thread quivers (defined below).  We start with a special case.

\begin{proposition}\label{proposition:PathCategory}
Let $\aa$ be a semi-hereditary dualizing $k$-variety without infinite threads.  Write $Q$ for the Auslander--Reiten quiver of $\aa$, then $\aa$ is equivalent to $kQ$.  Moreover, $Q$ is a strongly locally finite quiver.
\end{proposition}

\begin{proof}
For any two indecomposables $A,B \in \ind \aa$, fix a basis for $\Irr(A,B) \cong \rad(A,B) / \rad^2(A,B)$ and choose a corresponding set of irreducible maps $A \to B$.  This gives a map $f: Q_1 \to \Mor \aa$, mapping the arrows in $Q$ to the chosen irreducible maps in $\aa$.  This induces an essentially surjective functor $F:kQ \to \aa$; we will show that this maps lifts to an equivalence $kQ \to \aa$.

Since $\aa$ is semi-hereditary, the functor $F$ is faithful (this follows from Proposition \ref{proposition:ShIsLocal}).  Consequently, since $\aa$ is Hom-finite, the quiver $Q$ has no cycles.

Seeking a contradiction, assume that $F$ is not full.  Let $X,Y \in \ind kQ$ and $f \in \aa(FX,XY)$ such that $f$ does not lie in $F(kQ(X,Y))$.  By construction of $F$, we know that a right almost split morphism $M \to Y$ in $kQ$ will be mapped to a right almost split morphism $FM \to FY$ in $\aa$.  If the map $f:FX \to FY$ is not an isomorphism, then $f$ factors through $FM \to FY$ and hence there is an indecomposable direct summand $Y_1$ of $M$ such that the map $kQ(X,Y_1) \to \aa(FX,FY_1)$ is not surjective.

Continuing this procedure, we obtain a sequence of irreducible $\cdots \to Y_n \to Y_{n-1} \to \cdots \to Y_1 \to Y$ such that $kQ(X,Y_n) \to \aa(FX,FY_n)$ is not surjective, for each $n$.  We have the following possibilities.

Either $Y_n \cong X$ for some $n \in \bN$ (and the above procedure does not create an infinite path).  In this case, we know that $kQ(X,Y_n) \to \aa(FX,FY_n)$ is an isomorphism.  Indeed, it follows from Proposition \ref{proposition:ShIsLocal}) that both $kQ(X,Y_n)$ and $\aa(FX,FY_n)$ are finite--dimensional local hereditary rings, thus $kQ(X,Y_n) \cong \aa(FX,FY_n) \cong k$.  Since the map $kQ(X,Y_n) \to \aa(FX,FY_n)$ is nonzero, we know it is an isomorphism and in particular surjective.

The second possibility is that the above construction gives an infinite sequence $\cdots \to Y_n \to Y_{n-1} \to \cdots \to Y_1 \to Y$.  Since $\aa$ does not have infinite threads (by assumption) and $[FX,FY]$ has only finitely many nonthread objects (Proposition \ref{proposition:ThreadsInVarieties}), we know that there is a cycle in this sequence, thus $Y_n \cong Y_m$ for some $m,n \in \bN$.  By Proposition \ref{proposition:ShIsLocal} we know that the indecomposable objects on this cycle form a semi-hereditary category, and such a category cannot be Hom-finite.  We have found a contradiction.

We conclude that the functor $F: kQ \to \aa$ is fully faithful and essentially surjective, thus an equivalence.  Since $\aa$ is a dualizing $k$-variety, we know that $Q$ is strongly locally finite.
\end{proof}

The previous proposition does not hold if $\aa$ has infinite threads.  However, in the case of dualizing $k$-varieties, these categories are nicely behaved (Proposition \ref{proposition:ThreadsInVarieties}) such that an analog of Proposition \ref{proposition:PathCategory} may be proven when $\aa$ does contain infinite threads, by replacing the quiver $Q$ by a thread quiver.

A \emph{thread quiver} $Q$ consists of the following information:
\begin{itemize}
\item A quiver $Q_u=(Q_0,Q_1)$ where $Q_0$ is the set of vertices and $Q_1$ is the set of arrows.
\item A decomposition $Q_1 = Q_s \coprod Q_t$.  Arrows in $Q_s$ will be called \emph{standard arrows}, while arrows in $Q_t$ will be referred to as \emph{thread arrows}.
\item For every thread arrow $\alpha \in Q_t$, there is an associated linearly ordered set $\TT_\alpha$, possibly empty.
\end{itemize}

In drawing a thread quiver (cf. Figure \ref{fig:ThreadQuiver}), standard arrows will be represented by $\xymatrix@1{\bullet \ar[r]& \bullet}$, while thread arrows will be drawn as $\xymatrix@1{\bullet\ar@{..>}[r]&\bullet}$ labeled by the corresponding ordered set $\TT$.  If $\TT = \{1,2,3, \ldots, n\}$ with the normal ordering, then we will only write $n$ as a label; if $\TT = \emptyset$, then no label will be written.

\begin{figure}
\exa
$$\xymatrix{\bullet \ar[r]\ar[rd]&\bullet \ar@<-2pt>[d] \ar@{..>}@<2pt>[d]^{\bZ} & \bullet \ar@{..>}[l]_{3} \ar@{..>}[r] & \bullet\\
&\bullet}$$
\exb
$$\xymatrix{\bullet \ar[r]\ar[rd]&\bullet \ar@<-2pt>[d] \ar@<2pt>[d] & \bullet \ar[l] \ar[r] & \bullet\\
&\bullet}$$
\exc
\caption{An example of a thread quiver $Q$ (left) and the corresponding underlying quiver $Q_u$ (right)}
\label{fig:ThreadQuiver}
\end{figure}
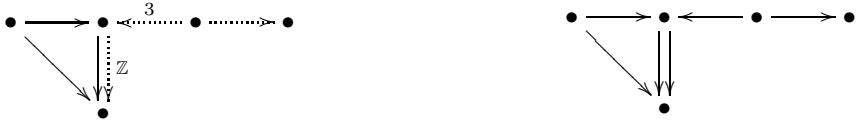

Every thread quiver $Q$ has an underlying quiver where there is no distinction between the standard and the thread arrows, and where all arrows are unlabeled.  To avoid confusion, we will refer to this underlying quiver by $Q_u$.

The following proposition will be useful in Construction \ref{construction:MakeThreadQuiver}.

\begin{proposition}\label{proposition:MakeThreadQuiver}
Let $\aa$ be a semi-hereditary dualizing $k$-variety.  Let $\{ [X_j,Y_j] \}_{j \in J}$ be the set of all maximal \emph{infinite} threads in $\aa$.  Let $\bb$ be the subcategory of $\aa$ consisting of all objects without direct summands in $\bigcup_{j \in J} \ind [X_j,Y_j]$.  Then $\bb$ is a semi-hereditary dualizing $k$-variety without infinite threads.
\end{proposition}

\begin{proof}
It follows from Corollary \ref{corollary:AllThreads} that $\bb$ is a semi-hereditary dualizing $k$-variety.  If we show that $\bb$ has no infinite threads, then the statement follows from Proposition \ref{proposition:PathCategory}.

Let $X,Y \in \ind \bb$ such that $[X,Y] \subseteq \bb$ is an infinite thread.  Since $X,Y \in \bb$, we know that $[iX,iY] \subseteq \aa$ is not an infinite thread, where $i: \bb \to \aa$ is the natural embedding.  Thus $[iX,iY]$ contains a nonthread object $Z \in \ind \aa$.  By the definition of $\bb$, we know that $Z \in \bb$; we will show that $Z \in \bb$ is a nonthread object.

Consider a minimal right almost split morphism $M \to Z$ in $\aa$.  Without loss of generality, assume that $M$ is not indecomposable.  Let $i_L: \aa \to \bb$ be a left adjoint to $i: \bb \to \aa$.  We will now show that $i_L M \to i_L Z \cong Z$ is a minimal right almost split morphism.  That it is a right almost split morphism follows from the properties of a left adjoint.  To see that $i_L M \to Z$ is minimal, let $M' \to i_L M$ be a split monomorphism such that the composition $M' \to i_L M \to Z$ is zero.  Since $\aa$ is semi-hereditary, we know that $i_l M \to Z$ is a monomorphism (here we use Proposition \ref{proposition:ShIsLocal} and that $M \to Z$ is minimal).  It then follows that $M' = 0$ and hence $i_L M \to Z$ is minimal.  Finally, $i_L M$ is not indecomposable if $M$ is not indecomposable.  Indeed, let $M_j$ be a direct summand of $M$.  If $M_j$ does not lie on a maximal infinite thread, then $i_L M_j \cong M_j$, and if $M_j \in [X_j,Y_j]$ for some maximal infinite thread $[X_j,Y_j]$, then $X_j^- \to M_j$ factors through $i_L M_j \to M_j$ such that $i_L M_j \not= 0$.

Thus $Z \in \bb$ is a nonthread object lying in $[X,Y] \subseteq \bb$, and thus $[X,Y]$ is not an infinite thread.
\end{proof}

Starting from a semi-hereditary dualizing $k$-variety, we will associate a thread quiver in the following way.

\begin{construction}\label{construction:MakeThreadQuiver}
Let $\aa$ be a semi-hereditary dualizing $k$-variety.  Let $\bb$ be as in Proposition \ref{proposition:MakeThreadQuiver}.  Proposition \ref{proposition:PathCategory} yields that $\bb \cong kQ_u$ for a strongly locally finite quiver $Q_u$.

According to Corollary \ref{corollary:ThreadsPosets}, every maximal infinite thread $[X_i,Y_i] \subset \ind \aa$ corresponds to a poset of the form $\bN \cdot (\TT \stackrel{\rightarrow}{\times} \bZ) \cdot -\bN$ where $\TT$ is a linearly ordered set.  We will replace one arrow from $\xymatrix@1{X^- \ar[r] & Y^+}$ in $Q_u$ by $\xymatrix@1{X^- \ar@{..>}[r]^{\TT} & Y^+}$.
\end{construction}

\begin{example}\label{example:Construction}
Let $\LL = \bN \cdot (\TT \stackrel{\rightarrow}{\times} \bZ) \cdot -\bN$, then the thread quiver associated with $k\LL$ is
$$\xymatrix@1{\cdot \ar@{..>}[r]^{\TT} & \cdot}$$
\end{example}

Given a thread quiver $Q$, we will construct an associated semi-hereditary finite $k$-variety as follows.

\begin{construction}\label{construction:ThreadQuiver}
Let $Q$ be a thread quiver with underlying strongly locally finite quiver $Q_u$.  With every thread $t \in Q_t$, we denote by $f^t: k(\cdot \to \cdot) \longrightarrow kQ_u$ the functor associated with the obvious embedding $(\cdot \to \cdot) \longrightarrow Q_u$.  We define the functor
$$f: \bigoplus_{t \in Q_t} k(\cdot \to \cdot) \longrightarrow kQ_u.$$

With every thread $t$, there is an associated linearly ordered set $\TT_t$.  We will write $\LL_t = \bN \cdot (\TT_t \stackrel{\rightarrow}{\times} \bZ) \cdot -\bN$ and denote by
$$g^t: k(\cdot \to \cdot) \longrightarrow k\LL_t$$
the $k$-linear functor induced by mapping the extremal points of $\cdot \to \cdot$ to the minimal and maximal objects of $\LL$, respectively.  We will write
$$g: \bigoplus_{t \in Q_t} k(\cdot \to \cdot) \longrightarrow \bigoplus_{t \in Q_t} k\LL_t.$$

We define the category $kQ$ as a 2-pushout in $\kVar$ of the following diagram.

$$\xymatrix{
\bigoplus_{t \in Q_t} k (\cdot \to \cdot) \ar[r]^-{f} \ar[d]_{g} & {k Q_u} \ar@{-->}^{i}[d] \\
\bigoplus_{t \in Q_t} k\LL_t \ar@{-->}[r]_-{j} & kQ
}$$
\end{construction}

\begin{remark}
The construction above is similar to the procedure outlined in Example \ref{example:PushOut}, where the category $A_4$ is replaced by an infinite linearly ordered set.
\end{remark}

\begin{remark}
Given a thread quiver $Q$, Construction \ref{construction:ThreadQuiver} will define $kQ$ only up to equivalence.
\end{remark}

\begin{example}
Let $Q$ be the thread quiver given by
$$\xymatrix@1{\cdot \ar@{..>}[r]^{\TT} & \cdot}$$
as in Example \ref{example:Construction}.  The associated finite $k$-variety is equivalent to $k(\bN \cdot \TT \stackrel{\rightarrow}{\times} \bZ \cdot -\bN)$.
\end{example}

\begin{remark}
Different thread quivers $Q,Q'$ may give rise to equivalent categories $kQ$ and $kQ'$ as shown in the following example.
\end{remark}

\begin{example}
The following three thread quivers give rise to equivalent categories.
$$\xymatrix@1{
\cdot \ar@{..>}[r]^{1} & \cdot & \cdot \ar@{..>}[r] & \cdot \ar@{..>}[r]& \cdot& \cdot \ar@{..>}[r]^{1} & \cdot \ar[r]& \cdot} $$
\end{example}

\begin{observation}
For each thread $t \in Q_t$, the functor $f^t$ has a left adjoint (denoted by $f^t_L$) and a right adjoint (denoted by $f^t_R$).  The induced functors
$$f_L, f_R : kQ_u \longrightarrow \bigoplus_{t \in Q_t} k(\cdot \to \cdot)$$
are left and right adjoints for $f$, respectively.  Note that $f^t$ will always be faithful, but in general not full.  Its adjoints $f^t_L$ and $f^t_R$ are in general not full nor faithful.

Similarly, for each thread $t \in Q_t$, the functor $g^t$ has a left adjoint (denoted by $g^t_L$) and a right adjoint (denoted by $g^t_R$), giving rise to the functors
$$g_L, g_R : \bigoplus_{t \in Q_t} k\LL_t \longrightarrow \bigoplus_{t \in Q_t} k(\cdot \to \cdot)$$
which are then left and right adjoints for $g$, respectively.  The functor $g$ is fully faithful and its adjoints $g_L$ and $g_R$ are faithful.
\end{observation}

\begin{example}\label{example:ftNotFull}
The functor $f^t$ is not full when $Q$ is given by the following thread quiver.
$$\xymatrix{\cdot \ar@<2pt>@{..>}[r] \ar@<-2pt>[r] & \cdot}$$
\end{example}

In what follows it will be convenient to assume the functors $f^t$ are fully faithful.  For this we will use the following lemma.

\begin{lemma}\label{lemma:ChooseGoodQuiver}
Let $Q$ be a thread quiver.  There is a thread quiver $Q'$ such that
\begin{enumerate}
\item $kQ \cong kQ'$, and
\item the functors ${f'}^t: k(\cdot \to \cdot) \longrightarrow kQ'_u$ from Construction \ref{construction:ThreadQuiver} are fully faithful.
\end{enumerate}
\end{lemma}

\begin{proof}
Define a thread quiver $Q'$ by replacing all thread arrows $\xymatrix{x \ar@{..>}[r]^\TT & y}$ in the thread quiver $Q$ by $\xymatrix{x \ar[r]& \cdot \ar@{..>}[r]^\TT & \cdot\ar[r] & y}$.  Since the associated functors ${f'}^t: k(\cdot \to \cdot) \longrightarrow kQ'_u$ are fully faithful by Proposition \ref{proposition:ThreadsInVarieties}, we need only to prove $kQ \cong kQ'$.

We start by defining a 2-pushout
$$\xymatrix{
\bigoplus_{t \in Q'_t} k (a_t \to b_t \to c_t \to d_t) \ar[r]^-{f_1} \ar[d]_{g_1} & {k Q'_u} \ar@{-->}[d] \\
\bigoplus_{t \in Q'_t} k\LL_t \ar@{-->}[r] & \pp
}$$
in $\kVar$ where the maps are given by the following descriptions.  The functor ${f_1}^t$, for each thread $t:\xymatrix{x_t \ar@{..>}[r]^\TT & y_t}$ in $Q'$, is given as a faithful functor which maps $b_t$ and $c_t$ to $x_t$ and $y_t$ respectively, and $a_t$ and $d_t$ to the direct predecessor of $x_t$ and the direct successor of $y_t$, respectively.

The functor $g_1^t$ is a fully faithful functor given by
\begin{eqnarray*}
a_t &\mapsto& 0_t \\
b_t &\mapsto& 1_t \\
c_t &\mapsto& -1_t \\
d_t &\mapsto& -0_t
\end{eqnarray*} 
where $0_t$ and $-0_t$ are the minimal and maximal element of $\LL_t$ respectively, $1_t$ is the direct successor of $0_t$, and $-1_t$ is the direct predecessor of $-0_t$.

Define a functor
$$H_1: \bigoplus_{t \in Q_t} k (x_t \to y_t) \longrightarrow \bigoplus_{t \in Q_t} k (a_t \to b_t \to c_t \to d_t)$$
by mapping $x_t$ to $a_t$, $y_t$ to $d_t$, and the generator $x_t \to y_t$ to the composition $a_t \to b_t \to c_t \to d_t$.  We also consider the natural embedding of categories $F: kQ_u \longrightarrow kQ_u'$.

Consider the following diagram where we write $A_2$ and $A_4$ for $(\cdot \to \cdot)$ and $(\cdot \to \cdot \to \cdot \to \cdot)$ respectively:

$$\xymatrix@+15pt{
\bigoplus_{t} k A_2 \ar[r]^f \ar[d]_{H_1} & kQ_u \ar[d]_F \ar@/^3pc/[dd]^i \\
\bigoplus_{t} k A_4 \ar[r]^{f_1} \ar[d]_{g_1}& kQ'_u \\
\bigoplus_{t} k\LL_t \ar[r]_j & kQ
}$$
where the outer diagram is a 2-pushout.  It is readily verified that the upper square is also a 2-pushout.  The 2-natural transformation given by
$$\mbox{$i: kQ_u \to kQ$ and $j \circ g_1: \bigoplus_{t} k A_4  \to \bigoplus_{t \in Q_t} k\LL_t \to kQ$}$$
induces a functor $i': kQ'_u \to kQ$.  Using the universal property, it is now straightforward to check that the lower square is also a 2-pushout.  This shows that $\pp \cong kQ$.

To show that $\pp \cong kQ'$, consider the 2-pushout 
$$\xymatrix@C+30pt{
\bigoplus_{t \in Q'_t} k (x'_t \to y'_t) \ar[r]^{f'} \ar[d]_{g'} & {k Q'_u} \ar@{-->}[d]^{i'} \\
\bigoplus_{t \in Q'_t} k\LL'_t \ar@{-->}[r]^{j'} & kQ'
}$$
where $\LL'_t = \bN_0 \cdot (\TT_t \stackrel{\rightarrow}{\times} \bZ) \cdot -\bN_0$ and $\bN_0= \bN \setminus \{0\}$.  Note that $\LL'_t \cong \LL_t$. 

We define a functor
$$H_2: \bigoplus_{t \in Q'_t} k (x'_t \to y'_t) \longrightarrow \bigoplus_{t \in Q'_t} k (a_t \to b_t \to c_t \to d_t)$$
by mapping $x'_t$ to $b_t$ and $y'_t$ to $c_t$.  We also consider the natural embedding of categories $G: \bigoplus_{t \in Q'_t} k\LL'_t \longrightarrow \bigoplus_{t \in Q'_t} k\LL_t$.

As before, one obtains a diagram
$$\xymatrix@C+25pt@R+10pt{
\bigoplus_{t} k A_2 \ar[r]^{H_2} \ar[d]_{g'} & \bigoplus_{t} k A_4 \ar[d]_{g_1} \ar[r]^-{f_1}& kQ'_u \ar[d]^i \\
\bigoplus_{t} k\LL'_t \ar[r]^-{G} \ar@/_2pc/[rr]_{j'}& \bigoplus_{t} k\LL_t & kQ'
}$$
where the left square and the outer diagram are 2-pushouts.  Again, one finds a morphism $\bigoplus_{t} k\LL_t \to kQ'$ making the right square a 2-pushout.  This shows that $\pp \cong kQ'$ and thus $kQ \cong kQ'$.
\end{proof}

\begin{example}
Let $Q$ be the quiver of Example \ref{example:ftNotFull}.  The quiver $Q'$ constructed in the proof of Lemma \ref{lemma:ChooseGoodQuiver} is given by
$$\xymatrix@R=5pt@C=25pt{ & \cdot \ar@{..>}[r]& \cdot \ar[rd] \\
\cdot \ar[rrr] \ar[ru] &&& \cdot}$$
\end{example}

\begin{remark}
The thread quiver $Q'$ will be strongly locally finite if and only if $Q$ is.
\end{remark}

\begin{proposition}\label{proposition:ijAdjoints}
The functor $i$ from Construction \ref{construction:ThreadQuiver} is fully faithful and has both a left and a right adjoint.  If $f^s$ is fully faithful for a thread arrow $s \in Q_t$, then the functor $j^s$ is fully faithful and has both a left and a right adjoint.
\end{proposition}

\begin{proof}
Since $g$ is fully faithful and has a left and a right adjoint, Theorem \ref{theorem:Pushout} shows the same properties hold for $i$.

Let $s \in Q_t$ be a thread arrow such that $f^s: k(\cdot \to \cdot) \longrightarrow kQ_u$ is fully faithful.  We will construct the left adjoint of $j^s: k\LL_s \to kQ$.  First, consider the 2-pushouts
\exa
$$\xymatrix@+10pt{k(\cdot \to \cdot) \ar[r]^{f^s} \ar[d]_{g^s} \xtwocell[rd]{}\omit{^<0>\alpha}& kQ_u \ar@{-->}[d]^{G}\\
k\LL_s \ar@{-->}[r]_{F^s} & \pp}$$
\exb
$$\xymatrix@+10pt{\bigoplus_{t \not= s} k(\cdot \to \cdot) \ar[r]^-{(G \circ f^t)_t}\ar[d]_{(g^t)_{t}} \xtwocell[rd]{}\omit{^<0>\beta} & \pp \ar@{-->}[d]^{i'}\\
\bigoplus_{t \not= s} k\LL_t \ar@{-->}[r] & \pp'}$$
\exc
Using the universal property, it is straightforward to proof that $\pp' \cong kQ$ and that $j^s = i' \circ F^s$.  Since both $f^s$ and $(g^t)_{t \not= s}$ are fully faithful and have a left and a right adjoint, Theorem \ref{theorem:Pushout} yields that the same properties hold for $F^s$ and $i'$, and hence also for $j^s$.
\end{proof}

\begin{lemma}\label{lemma:jAdmitsAdjoints}
Let $Q$ is a thread quiver with a strongly locally finite underlying quiver $Q_u$.  If $f^s$ is fully faithful for every thread arrow $s \in Q_t$, then the functor $j: \bigoplus_{t \in Q_t} k\LL_t \to kQ$ has a left and a right adjoint.
\end{lemma}

\begin{proof}
It follows from Proposition \ref{proposition:ijAdjoints} that each functor $j^t: k\LL_t \to kQ$ has a left adjoint functor $j_L^t$ and a right adjoint functor $j_R^t$.  We want to show that the functor $j_R: kQ \to \bigoplus_{t \in Q_t} k\LL_t$ given by $j_R(A) = \oplus_t j_R^t(A)$ is a right adjoint.  To show this, it suffices to show that for every $A \in kQ$ we have $j_R^t(A) = 0$ for all but finitely many thread arrows $t \in Q_t$.

By the Yoneda lemma, we know that $j_R^t(A) \not= 0$ if and only if $k\LL_t(-,j_R^t(A)) \not= 0$ and thus if and only if $k\LL(0_t, j_R^t(A)) \not= 0$, where $0_t \in \LL_t$ is the unique minimal element.  Let $x_t \in k(\cdot \to \cdot)$ such that $g^t(x_t) \cong 0_t$ (this object exists by the definition of $g^t$ as given in Construction \ref{example:Construction}).

Using that $j^t \circ g^t \cong i \circ f^t$, we find that $j_R^t(A) \not= 0$ if and only if $kQ(j^t \circ g^t (x_t), A) \cong kQ(i \circ f^t(x_t), A) \not= 0$.  Since $Q_u$ is strongly locally finite, we know there are only finitely many indecomposables $X \in \ind kQ_u$ such that $kQ(i(X),A) \cong kQ_u(X,i_R A)\not= 0$.  From this then follows that $j_R^t(A) \not= 0$ for only finitely many thread arrows $t \in Q_t$, and we conclude that $j: \bigoplus_{t \in Q_t} k\LL_t \to kQ$ has a right adjoint.

The proof that $j$ admits a left adjoint is similar.
\end{proof}

\begin{example}
Let $Q$ be the thread quiver given by $\xymatrix@1{\cdots \ar[r] & \ar@{..>}[r] & \ar[r] & \ar@{..>}[r] & \ar[r] & \cdots}$, thus the underlying quiver is a linearly orientated $A_\infty$-quiver and the arrows are alternatingly standard arrows and thread arrows.  In this case, the underlying quiver $Q_u$ is not strongly locally finite and the functor $j$ does not admit a right adjoint.
\end{example}

%
\begin{proposition}\label{proposition:Representable}
Let $Q$ is a thread quiver with a strongly locally finite underlying quiver $Q_u$, and assume that $f^s$ is fully faithful for every thread arrow $s \in Q_t$.  Let $\eta:M \to N$ be a morphism in $\Mod kQ$.  Then
\begin{enumerate}
\item $\eta: M \to N$ is an epimorphism (monomorphism) if and only if both $\eta \bullet 1_i: M \circ i \to N \circ i$ and $\eta \bullet 1_j: M \circ j \to N \circ j$ are epimorphisms (monomorphisms),
\item a representation $M \in \Mod kQ$ is finitely generated if and only if the restrictions $M \circ i$ and $M \circ j$ are finitely generated, and 
\item a representation $M \in \Mod kQ$ is finitely presented if and only if the restrictions $M \circ i$ and $M \circ j$ are finitely presented.
\end{enumerate}
\end{proposition}

\begin{proof}
First, if $\eta: M \to N$ is an epimorphism then so are $\eta \bullet 1_i$ and $\eta \bullet 1_j$ since restriction is an exact functor.  For the other direction, let $\mu: N \to C$ be a cokernel of $F$.  Then $(\mu \circ \eta) \bullet 1_i = (\mu \bullet 1_i) \circ (\eta \bullet 1_i) = 0$ and hence $\mu \bullet 1_i = 0$ since $\eta \bullet 1_i$ is an epimorphism.  Similarly, we show that $\mu \bullet 1_j = 0$, and hence $\mu: N \to C$ is the zero map.  We conclude that $\eta: M \to N$ is an epimorphism.

Next, let $M \in \Mod kQ$ be finitely generated, thus there is an epimorphism $kQ(-,X) \to M$.  Since restriction is an exact functor, there is an epimorphism $kQ(i-,X) \to M \circ i$, and hence $kQ(-,i_R X) \to M \circ i$.  This shows that $M \circ i$ is finitely generated.  Similarly, one shows that $M \circ j$ is finitely generated.

For the other direction, assume that both $M \circ i$ and $M \circ j$ are finitely generated, thus there are epimorphisms $kQ_u(-,A) \to M \circ i$ and $k\LL(-,B) \to M \circ j$, for some $A \in kQ_u$ and $B \in k\LL$.  Using the first part, one checks that the natural map $kQ(-,iA) \oplus kQ(-,jB) \to M$ is an epimorphism.  This shows that $M$ is finitely generated.

For the last statement, the proof that when $M$ is finitely presented, then so are $M \circ i$ and $M \circ j$ is similar as in the finitely generated case.  Thus assume that both $M \circ i$ and $M \circ j$ are finitely presented.  We already know that $M$ is finitely generated, thus there is an epimorphism $kQ(-,X) \to M$.  Let $K$ be the kernel.  Using again that restriction is exact, we know that $K \circ i$ is the kernel of $kQ(i-,X) \to M \circ i$ and that $K \circ j$ is the kernel of $kQ(j-,X) \to M \circ j$, and thus both $K \circ i$ and $K \circ j$ are finitely generated.  We conclude that $K$ is finitely generated and hence $M$ is finitely presented.

\end{proof}

The following proposition gives sufficient conditions for Construction \ref{construction:ThreadQuiver} to give a semi-hereditary dualizing $k$-variety.


\begin{proposition}\label{proposition:kQIsGood}
If $Q$ is a thread quiver with a strongly locally finite underlying quiver $Q_u$, then $kQ$ is a semi-hereditary dualizing $k$-variety.
\end{proposition}

\begin{proof}
We will assume that the quiver $Q$ satisfies the property of Lemma \ref{lemma:ChooseGoodQuiver}.   It follows from Proposition \ref{proposition:ijAdjoints} that $i$ has a left adjoint $i_L$ and a right adjoint $i_R$, and it follows from Lemma \ref{lemma:jAdmitsAdjoints} that $j$ has a left adjoint $j_L$ and a right adjoint $j_R$.  A finitely generated projective module $kQ(-,A) \in \Mod kQ$ thus corresponds to a diagram
$$\xymatrix{
\bigoplus_{t} k (\cdot \to \cdot) \ar[r]^-{f} \ar[d]_{g} \xtwocell[rd]{}\omit{^<0>\alpha}& {k Q_u} \ar[d]^{i} \ar@/^2pc/[rdd]^{kQ_u(-,i_R A)} \xtwocell[rdd]{}\omit{<0>\gamma}\\
k\LL \ar[r]_-{j} \ar@/_2pc/[rrd]_{\LL(-,j_R A)} \xtwocell[rrd]{}\omit{^<0>\beta}& kQ\ar@{-->}[rd]|-{kQ(-,A)} \\
&& \Mod k}$$
where we have written $k\LL = \bigoplus_{t} k \LL_t$.  Using that the restriction functors are exact, this shows that the natural embedding $\Mod kQ \to \Mod kQ_u \oplus \Mod k\LL$ induces a faithful functor $\mod kQ \to \mod kQ_u \oplus \mod k\LL$.  Since the latter category is Hom-finite, so is the former, and hence $kQ$ is a finite $k$-variety.

To show that $kQ$ is semi-hereditary (thus that $\mod kQ$ is abelian and hereditary), we will use Proposition \ref{proposition:ShIsLocal}.  Thus let $X \in kQ$ be any object.  We will show that there is a semi-hereditary full subcategory of $kQ$ which contains $X$.

For each thread arrow $t \in Q_t$, let $\LL'_t$ be the (finite) subposet of $\LL_t$ containing the minimum and the maximum of $\LL_t$, together with the elements corresponding to the indecomposable direct summands of $j_R^t(X)$.  The inclusion $\LL'_t \subset \LL_t$ induces a fully faithful functor $k\LL'_t \to k\LL_t$ and $j_R^t(X)$ lies in the essential image of that functor.  Also note that $k\LL'_t \cong kA_{n_t}$ for some $n_t \geq 2$.  We define $kQ'$ as the 2-pushout of the diagram
$$\xymatrix{
\bigoplus_{t} k (\cdot \to \cdot) \ar[r] \ar[d] & {k Q_u}\\
\bigoplus_{t} k\LL'_t}$$
Using the universal property of the 2-pushout, the functors $k\LL'_t \to k\LL_t \to kQ$ and $kQ_u \stackrel{\sim}{\rightarrow} kQ_u \to kQ$ induce a functor $kQ' \to kQ$.  Moreover, one easily checks that $kQ'$ is the $k$-linear additive path category of the quiver $Q'$ given by replacing the thread arrow $t \in Q$ by the corresponding $A_{n_t}$-quiver (this can be done using the universal property of the 2-pushout, or using the explicit construction as was done in Example \ref{example:PushOut})

To check that $kQ' \to kQ$ is fully faithful, it suffices to check that the induced functor $-\otimes_{kQ'} kQ:\Mod kQ' \to \Mod kQ$ is fully faithful on projectives; it then follows from Lemma \ref{lemma:TensorIsFF} that the functor $kQ' \to kQ$ is fully faithful.  The verification is straightforward, using that $-\otimes_{k\LL'_t} k\LL_t$ is fully faithful for each $t \in Q_t$.

Note that an object $M \in \Mod kQ$ lies in the essential image of $-\otimes_{kQ'} kQ$ if and only if $M \circ j^t \in \mod k\LL_t$ lies in the essential image of $-\otimes_{k\LL'_t} k\LL_t$ for each $t \in Q_t$.  Hence $kQ(-,X)$ lies in the essential image of $-\otimes_{kQ'} kQ$ and thus $X$ lies in the essential image of $kQ' \to kQ$.  We have shown that $X$ lies in a semi-hereditary full subcategory of $kQ$, and hence it follows that $kQ$ is semi-hereditary.

To show that $kQ$ is a dualizing $k$-variety, it suffices to show that standard projectives are cofinitely presented and standard injectives are finitely presented (see Proposition \ref{proposition:Dualizing}).  Let $X \in kQ$ and let $kQ(X,-)^*$ be a standard injective.  Since the functors $i: kQ_u \to kQ$ and $j: k\LL \to kQ$ admit left adjoints, the restrictions $kQ(X,i-)^*$ and $kQ(X,j-)^*$ of $kQ(X,-)^*$ are standard injectives as well.  Since $kQ_u$ and $k\LL$ are dualizing $k$-varieties, these objects are finitely presented and thus, by Proposition \ref{proposition:Representable}, so is $kQ(X,-)^*$.  Dually, one can show that the standard projectives are cofinitely presented.  We conclude that $kQ$ is a dualizing $k$-variety.

\end{proof}

We are now ready to complete the classification of semi-hereditary dualizing $k$-varieties by thread quivers.

\begin{theorem}\label{theorem:ThreadQuivers}
Let $\aa$ be a semi-hereditary dualizing $k$-variety, then $\aa$ may be obtained by Construction \ref{construction:ThreadQuiver} from a thread quiver $Q$ where $Q$ is strongly locally finite.
\end{theorem}

\begin{proof}
It has been established in Proposition \ref{proposition:kQIsGood} that $kQ$ is indeed a semi-hereditary dualizing $k$-variety, so that we need only to show that every semi-hereditary dualizing $k$-variety arises in this way (up to equivalence).

Let $\aa$ be a semi-hereditary dualizing $k$-variety and let $Q$ be the corresponding thread quiver as in Construction \ref{construction:MakeThreadQuiver}.  We will work with the thread quiver constructed in Lemma \ref{lemma:ChooseGoodQuiver}, so that we can apply Proposition \ref{proposition:ijAdjoints} and Lemma \ref{lemma:jAdmitsAdjoints}.  We will show that $\aa \cong kQ$.

There are obvious embeddings $i^\aa: kQ_u \to \aa$ and $j^{t,\aa}: k\LL_t \to \aa$ giving the following diagram
$$\xymatrix{
\bigoplus_{t} k (\cdot \to \cdot) \ar[r]^-{f} \ar[d]_{g} \xtwocell[rd]{}\omit{^<0>\alpha}& {k Q_u} \ar[d]^{i} \ar@/^2pc/[rdd]^{i^\aa} \xtwocell[rdd]{}\omit{<0>}\\
\oplus_t k\LL_t \ar[r]_-{j} \ar@/_2pc/[rrd]_{j^\aa} \xtwocell[rrd]{}\omit{^<0>}& kQ\ar@{-->}[rd]_{F} \\
&& \aa}$$
We want to show that the functor $F: kQ \to \aa$ is an equivalence.  Since the induced functor $kQ_u \oplus k\LL \to \aa$ is essentially surjective, so is the functor $F$.  We will continue by showing $F$ is also fully faithful.  First note that $i^\aa: kQ_u \to \aa$ has a right adoint $i^\aa_R$ and a left adjoint $i^\aa_L$ by Proposition \ref{proposition:AllThreads}, and that $j^{\aa,t}: k\LL_t \to \aa$ has a right adoint $j^{\aa,t}_R$ and a left adjoint $j^{\aa,t}_L$ by Proposition \ref{proposition:WhenAdjoints}.

Since $i^\aa \cong F \circ i$ and both $i$ and $i^\aa$ are fully faithful, $F$ induces a bijection $kQ(iX,iY) \to \aa(FiX, FiY)$, for each $X,Y \in kQ_u$.

Likewise, for each $t \in Q_t$, the functors $j^t$ and $j^{\aa,t}$ are fully faithful, and hence $F$ induces a bijection $kQ(j^t X,j^t Y) \to \aa(F j^t X, F j^t Y)$, for each $X,Y \in k\LL_t$.

We will first show that $F: kQ \to \aa$ is faithful.  Let $h \in kQ(X,Y)$ for any $X,Y \in \ind kQ$; it suffices to show that $Fh \not= 0$.  Since $kQ$ is semi-hereditary (Proposition \ref{proposition:kQIsGood}), we may invoke Proposition \ref{proposition:ShIsLocal} to see that $h$ is a monomorphism when $h \not= 0$.  Let $\epsilon: i \circ i_R \to 1_{kQ}$ be the counit of the adjunction.  To show that $h \circ \epsilon_X: i \circ i_R(X) \to X \to Y$ is nonzero, it thus suffices to show that $\epsilon_X: i \circ i_R(X) \to X$ is nonzero.

Since $X$ is nonzero, so is $kQ(-,X)$ and by the universal property of the 2-pushout, the functors $kQ(i-,X) \cong kQ(-,i_R X)$ and $kQ(j^t-,X) \cong kQ(-,j^t_R X)$ are not all zero (ranging over all $t \in Q_t$).  If $i_R X \not\cong 0$, then neither is $\epsilon_X: i \circ i_R(X) \to X$ since $i$ is fully faithful.  Thus assume that $i_R X \cong 0$ and thus there is a $t \in Q_t$ such that $j_R^t X \not\cong 0$ and hence $g_R^t \circ j_R^t X \not\cong 0$.  We then find that $g_R \circ j_R \not\cong 0$ and since both $f_R \circ i_R$ and $g_R \circ j_R$ are right adjoint to $j \circ g \cong i \circ f$, we find that $i_R X \not\cong 0$.  Contradiction.

We may thus conclude that $h \circ \epsilon_X \not= 0$.  Using that $Y$ is indecomposable, and hence $h \circ \epsilon_X$ is an epimorphism, we may similarly prove that $\eta_Y \circ h \circ \epsilon_X: i \circ i_R X \to i \circ i_L Y$ is nonzero, where $\eta: 1_{kQ} \to i \circ i_L$ is the unit of the adjunction.

Since $i: kQ_u \to kQ$ is fully faithful, there is a map $h': i_R X \to i_L Y$ such that $\eta_Y \circ h \circ \epsilon_X = i(h')$.  Since $i^{\aa}(h') \not= 0$, we may conclude that $F(\eta_Y \circ h \circ \epsilon_X) = F(\eta_Y) \circ F(h) \circ F(\epsilon_X) \not= 0$ and hence $F(h) \not= 0$.  Hence $F$ is faithful.

We will now show that $F$ is full.  Let $X,Y \in kQ$ and $h \in \aa(FX,FY)$.  We want to show that there is a map $h' \in kQ(X,Y)$ such that $F(h') = h$.  If $X,Y$ both lie in the essential image of either $i$ or $j^t$, for some $t \in Q_t$, then the statement follows easily from $i$ or $j^t$ being full, respectively.

We will consider the case where $X$ lies in the essential image of $j^t$ (for some $t \in Q_t$) and $Y$ lies in the essential image of $i$.  The other cases are handled in a similar way.  Without loss of generality, we may assume that there is an $X' \in k\LL_t$ and a $Y' \in kQ_u$ such that $j^t(X') = X$ and $i(Y') = Y$.

By Corollary \ref{corollary:MapsOutOfThreads}, we know that there is an indecomposable object $A \in \aa$, lying in the essential image of both $i^\aa$ and $j^{\aa,t}$,  such that the map $h:FX \to FY$ factors as $FX \to A \to FY$.  Moreover, we can write $h = h_2 \circ \xi \circ h_1$ where $h_1: FX \to j^{\aa,t} A_1$, $h_2: i^\aa A_2 \to FY$, and $\xi: j^{\aa,t} A_1 \stackrel{\sim}{\rightarrow} i^{\aa} A_2$.  Since $i^\aa \cong F \circ i$ and $j^{\aa,t} \cong F \circ j^t$, we may assume that both $j^{\aa,t} A_1$ and $i^\aa A_2$ lie in the image of $F$.  Since each of these maps then lies in the image of $F$, so does the composition.  This implies that $F$ is full, and concludes the proof.

\end{proof}

\section{Representations of thread quivers}\label{section:Representations}

Let $Q$ be a thread quiver.  We define the categories $\Rep_k Q$ and $\rep_k Q$ to be $\Mod kQ$ and $\mod kQ$, respectively.  We have the following result.

\theoremB*

\begin{proof}
For the first statement, recall from Proposition \ref{proposition:kQIsGood} that $kQ$ is a semi-hereditary dualizing $k$-variety.  The statement then follows from Corollary \ref{corollary:ShDualizing}.

We now turn to the second statement.  Let $\aa$ be the category of projectives of $\AA$, thus $\aa$ is a finite $k$-variety.  The embedding $\aa \to \AA$ lifts to a fully faithful right exact functor $\mod \aa \to \AA$, which is an equivalence since $\AA$ has enough projectives.  By definition $\aa$ is semi-hereditary and Theorem \ref{theorem:RepSerreDuality} yields that $\aa$ is a dualizing $k$-variety.  The result now follows from Theorem \ref{theorem:ThreadQuivers}.
\end{proof}

With a thread quiver $Q$, there is thus an associated 2-functor $\aa: \Free(\cdot \leftarrow \cdot \rightarrow \cdot) \longrightarrow \kVar$ given by
$$\xymatrix{
\bigoplus_{t \in Q_t} k (\cdot \to \cdot) \ar[r]^-{f} \ar[d]_{g} & {k Q_u} \\
\bigoplus_{t \in Q_t} k\LL_t
}$$
The category $\Rep_k Q$ is equivalent to $\Hom_{2\FF}(\aa,\Mod k)$.  Thus the objects of the category $\Rep_k Q$ are given by
\begin{enumerate}
\item an object $N(-):kQ_u \to \Mod k$ of $\Mod kQ_u$,
\item for every thread $t$, an object $L_t(-):\LL_t \to \Mod k$ of $\Mod k\LL_t$, and
\item a natural equivalence $\alpha: \oplus_t L_t(g-) \Rightarrow N(f-)$.
\end{enumerate}
The morphisms are given by the modifications, thus given the data $(N,\{L_t\}_t, \alpha)$ and $(N',\{L'_t\}_t, \alpha')$ of two representations in $\Rep_k Q$, a morphism is given by
\begin{enumerate}
\item a natural transformation $\beta: N \Rightarrow N'$,
\item a natural transformation $\gamma: \oplus_t L_t \Rightarrow \oplus_t L'_t$
\end{enumerate}
such that the following diagram commutes.
$$\xymatrix{
\oplus_t L_t(g-) \ar[r]^-\alpha \ar[d]_{\gamma \bullet 1_g} & N(f-) \ar[d]^{\beta \bullet 1_f}\\
\oplus_t L'_t(g-) \ar[r]_-{\alpha'} & N'(f-)
}$$

As in the proof of Proposition \ref{proposition:kQIsGood}, the category $\rep_k Q$ has a similar description obtained by requiring $N$ and $L_t$ to be finitely presented representations.  If $Q$ is a strongly locally finite quiver, then $\rep_k Q$ has Serre duality.

Note that even when $\mod kQ_u$ and $\mod \left( \oplus_{t \in Q_t} k\LL_t \right)$ are well-understood, the last commutative diagram that is required can make the category $\rep_k Q$ contain a wild hereditary category.

\begin{example}
The category of finitely generated representations of the thread quiver
$$\xymatrix@R=15pt@C=25pt{&&\cdot \\
\cdot \ar[r] & \cdot \ar[r] & \cdot \ar@{..>}[r] \ar[u] & \cdot}$$
contains the representations of a wild quiver as a full and exact subcategory (see Proposition \ref{proposition:SemiHereditaryExact}).
\end{example}

As another way to describe the category of representations of thread quivers, we give the following proposition.

\begin{proposition}
Let $\aa$ be a finite $k$-variety.  Let $\{\aa_i\}_i$ be the filtered system of all subcategories of $\aa$ such that $\ind \aa_i$ is finite.  Let $\{\mod \aa_i\}_i$ be the filtered system where the maps are given by $\aa_j \otimes_{\aa_i} -: \mod \aa_i \to \mod \aa_j$ whenever $\aa_i \subseteq \aa_j \subseteq \aa$.  There is an equivalence
$$2\lim\limits_{\stackrel{\rightarrow}{i}} \mod \aa_i \cong \mod \aa$$
induced by $- \otimes_{\aa_i} \aa: \mod \aa_i \to \mod \aa$.

When $\aa$ is semi-hereditary, these functors $- \otimes_{\aa_i} \aa$ are exact.
\end{proposition}

\begin{proof}
It follows from Lemma \ref{lemma:TensorIsFF} that the functors $- \otimes_{\aa_i} \aa$ are fully faithful, hence the functor $2\lim\limits_{\rightarrow} \mod \aa_i \to \mod \aa$ is fully faithful as well.  It is easy to see that the functor is essentially surjective, hence it is an equivalence.

If $\aa$ is semi-hereditary, then Proposition \ref{proposition:SemiHereditaryExact} yields that the functors $- \otimes_{\aa_i} \aa$ are exact.
\end{proof}

\begin{remark}
The above proposition also holds when the 2-colimit is taken in the 2-category $\Cat$ of all small categories (see \cite[Proposition A.5.5]{Waschkies04}).
\end{remark}

\begin{example}
Let $Q$ be either
\begin{enumerate}
\item a Dynkin quiver, or
\item a thread quiver of the form $\xymatrix@R=15pt@C=25pt{\cdot \ar@{..>}[r]^{\PP}& \cdot}$ or $\xymatrix@R=15pt@C=25pt{\cdot \ar[r]& \cdot \ar@{..>}[r]^{\PP}& \cdot\\ & \cdot \ar[u]}$,
\end{enumerate}
where $\PP$ is a linearly ordered poset, then $\rep Q$ is a directed hereditary category with Serre duality.  It has been shown in \cite{vanRoosmalen06} that all directed hereditary categories with Serre duality are, up to derived equivalence, of this form.
\end{example}

\begin{example}
Let $Q$ be the thread quiver given by
$$\xymatrix{x \ar@<2pt>@{..>}[r] \ar@<-2pt>[r] & y}$$
The category $\rep Q$ has been discussed in \cite{vanRoosmalen09}, where it is shown that $\rep Q$ contains a so-called big tube (defined in \cite{vanRoosmalen09}).

By replacing the thread arrow $\xymatrix{\ar@{..>}[r]&}$ by a linearly ordered $A_n$ quiver (where $n \geq 2$), we obtain a quiver $Q'$ whose path algebra $kQ'$ is a hereditary canonical algebra (as defined in \cite{Ringel84}) and
the category of finite dimensional modules over this quiver can be interpreted (up to derived equivalence) as the category of coherent sheaves on a weighted projective line of weight type $(n-2)$ (see \cite{GeigleLenzing87}, see also \cite{Lenzing07} for a more recent account of weighted projective lines).  Similarly, the category $\rep_k Q$ can then be thought of as being derived equivalent to the category of coherent sheaves on a weighted projective line of weight type $(\infty)$.
\end{example}

%% file: Appendix.tex
\appendix
\section{2-Pushouts}\label{Appendix}

Let $\II$ be the 1-category $\Free(2 \stackrel{s}{\leftarrow} 1 \stackrel{t}{\rightarrow} 3)$ and $\aa: \II \to \kVar$ a 2-functor.  The 2-colimit of $\aa$ is called a \emph{2-pushout}.

To specify a 2-natural transformation from $\aa$ to $\CC \in \kVar$, one needs three functors ($\ff_i: \aa(i) \to \CC$, i=1,2,3) and two natural equivalences ($\ff_1 \stackrel{\sim}{\Rightarrow} \ff_2 \circ \aa(s)$ and $\ff_1 \stackrel{\sim}{\Rightarrow} \ff_3 \circ \aa(t)$).  Up to invertible modification, we may assume that $\ff_1 = \ff_2 \circ \aa(s)$ so that it suffices to give only the functors $\ff_2$ and $\ff_3$, and a natural equivalence $\ff_2 \circ \aa(s) \stackrel{\sim}{\Rightarrow} \ff_3 \circ \aa(t)$.

The following theorem shows a connection between 2-pushouts and adjoints.

\begin{theorem}\label{theorem:Pushout}
Consider the 2-pushout
$$\xymatrix{\aa \ar[r]^f \ar[d]_g \xtwocell[rd]{}\omit{^<0>\alpha}& \bb \ar@{-->}[d]^i\\
\cc \ar@{-->}[r]_j & \pp}$$
in any strict 2-category.  If $g$ is fully faithful and has a left (right) adjoint, then $i$ is fully faithful and has a left (right) adjoint.
\end{theorem}

\begin{proof}
Consider the first diagram in Figure \ref{fig:LeftAdjointOfPsi}, where the natural transformation is given by
$$1_{f} \bullet \epsilon_g: f \circ g_L \circ g \stackrel{\sim}{\Rightarrow} f,$$
and where $\epsilon_g: g_L \circ g \Rightarrow 1$ is the counit of the adjunction $(g_L,g)$. We can complete this diagram as shown in Figure \ref{fig:LeftAdjointOfPsi}.  This gives the following identity
\begin{equation}\label{equation:AlphaBetaGamma}
(\gamma^{-1} \bullet 1_f) \circ (1_{i_L} \bullet \alpha) \circ (\beta \bullet 1_g) = 1_f \bullet \epsilon_g.
\end{equation}

We claim $i_L: \pp \to \bb$ is left adjoint to $i$.  To prove this, we shall define a unit $\eta:  1_{\pp} \Rightarrow i \circ i_L$ and a counit  $\epsilon: i_L \circ i \Rightarrow 1_{\bb}$ and show that
$$(1 \bullet \epsilon)\circ(\eta \bullet 1) = 1 \mbox{ and } (\epsilon \bullet 1)\circ(1 \bullet \eta) = 1.$$

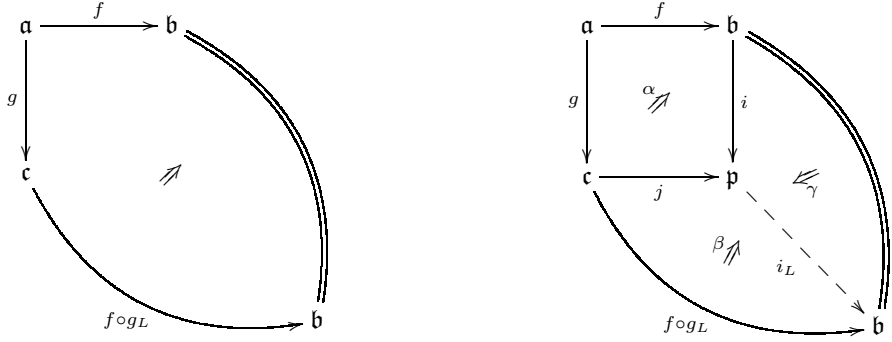
\begin{figure}
	\centering
	\exa
	$$\xymatrix@+20pt{
\aa \ar[r]^-{f} \ar[d]_{g} \xtwocell[rrdd]{}\omit{^<0>}& {\bb} \ar@/^2pc/@{=}[rdd] \\
\cc \ar@/_2pc/[rrd]_{f \circ g_L} \\
&& {\bb}}$$
	\exb
		$$\xymatrix@+20pt{
\aa \ar[r]^-{f} \ar[d]_{g} \xtwocell[rd]{}\omit{^<0>\alpha}& {\bb} \ar[d]^{i} \ar@/^2pc/@{=}[rdd] \xtwocell[rdd]{}\omit{<0>\gamma}\\
\cc \ar[r]_-{j} \ar@/_2pc/[rrd]_{f \circ g_L} \xtwocell[rrd]{}\omit{^<0>\beta}& \pp\ar@{-->}[rd]_{i_L} \\
&& {\bb}}$$
   \exc
	\caption{Construction of $i_L$}
	\label{fig:LeftAdjointOfPsi}
\end{figure}

\begin{figure}
	\centering
	\exa
		$$\xymatrix@+20pt{
\aa \ar[r]^-{f} \ar[d]_{g} \xtwocell[rd]{}\omit{^<0>\alpha} & {\bb} \ar@{-->}[d]^{i} \ar@/^2pc/[rdd]^{i} \xtwocell[rdd]{}\omit{=} \\
\cc \ar@{-->}[r]_-{j} \ar@/_2pc/[rrd]_{j} \xtwocell[rrd]{}\omit{=} & \pp\ar@{=}[rd] \\
&& {\pp}}$$
  \exb
		$$\xymatrix@+20pt{
\aa \ar[r]^-{f} \ar[d]_{g} \xtwocell[rd]{}\omit{^<0>\alpha} & {\bb} \ar@{-->}[d]^{i} \ar@/^2pc/[rdd]^{i} \xtwocell[rdd]{}\omit{<0> 1 \bullet \gamma} \\
\cc \ar@{-->}[r]_-{j} \ar@/_2pc/[rrd]_{j \circ g \circ g_L} \xtwocell[rrd]{}\omit{^<0>\delta}& \pp\ar@{-->}_{i \circ i_L}[rd] \\
&& {\pp}}$$
  \exc
	\caption{Construction of unit}
	\label{fig:UnitOfPsi}
\end{figure}
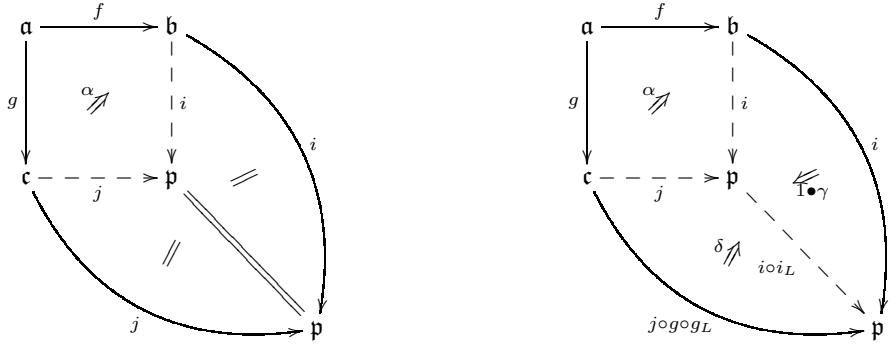

We define $\epsilon = \gamma^{-1}: i_L \circ i \stackrel{\sim}{\Rightarrow} 1_{\bb}$ with $\gamma$ as in the second diagram in Figure \ref{fig:LeftAdjointOfPsi}.  For the definition of $\eta$, consider the diagrams in Figure \ref{fig:UnitOfPsi}, where the natural transformations on the right-hand side are given by
\begin{eqnarray*}
1_{i} \bullet \gamma &:& i \circ 1_\bb \Rightarrow i \circ i_L \circ i \\
\delta=(1_{i} \bullet \beta) \circ (\alpha \bullet 1_{g_L}) &:& j \circ g \circ g_L \Rightarrow i \circ f \circ g_L \Rightarrow i \circ i_L \circ j
\end{eqnarray*}
There is a modification, going from the 2-natural transformation on the left-hand side of Figure \ref{fig:2ndEquation} to the one on the right-hand side, given by
\begin{eqnarray*}
1_{i} &:& i \Rightarrow i \\
1_{j} \bullet \eta_g&:& j \circ 1_\cc \Rightarrow j \circ g \circ g_L
\end{eqnarray*}
where $\eta_g: 1 \Rightarrow g \circ g_L$ is the unit of the adjunction $(g_L, g)$.  This modification corresponds to a natural transformation $\eta: 1_{\pp} \Rightarrow i \circ i_L$, satisfying the following commutative diagrams (see Remark \ref{remark:Modifications})

$$\xymatrix@C+20pt@R+5pt{
i \ar@{=}[r]\ar@{=}[d] & i \ar@{=>}[d]^{1 \bullet \gamma} & {j} \ar@{=>}[r]^-{1 \bullet \eta_{g}} \ar@{=}[d] & {j \circ g \circ g_L}\ar@{=>}[d]^{\delta} \\
i \ar@{=>}[r]_-{\eta \bullet 1} & {i \circ i_L \circ i} & {j} \ar@{=>}[r]_-{\eta \bullet 1_j} & i \circ i_L \circ j}$$

From the first commutative diagram, we obtain
$$(1 \bullet \epsilon)\circ(\eta \bullet 1) = (1 \bullet \gamma)^{-1}\circ(\eta \bullet 1) = 1.$$

\begin{figure}
	\centering
	\exa
		$$\xymatrix@+20pt{
\aa \ar[r]^-{f} \ar[d]_{g} \xtwocell[rd]{}\omit{^<0>\alpha} & {\bb} \ar@{-->}[d]^{i} \ar@/^2pc/[rrddd]^{i_L \circ i} \xtwocell[rrddd]{}\omit{=} \\
\cc \ar@{-->}[r]_-{j} \ar@/_2pc/[rrrdd]_{i_L \circ j} \xtwocell[rrrdd]{}\omit{=} & \pp \ar@{=}[rd] \\
&& {\pp} \ar[rd]|-{i_L} \\
&&& \bb}$$
  \exb
		$$\xymatrix@+20pt{
\aa \ar[r]^-{f} \ar[d]_{g} \xtwocell[rd]{}\omit{^<0>\alpha} & {\bb} \ar@{-->}[d]^{i} \ar@/^2pc/[rrddd]^{i_L \circ i} \xtwocell[rrddd]{}\omit{<-1>} \\
\cc \ar@{-->}[r]_-{j} \ar@/_2pc/[rrrdd]_{i_L \circ j \circ g \circ g_L} \xtwocell[rrrdd]{}\omit{^<0>} & \pp\ar[rd]|-{i \circ i_L} \\
&& {\pp} \ar[rd]|-{i_L} \\
&&& \bb}$$
  \exc
	\caption{Diagrams occurring in the definition of the modification $\Lambda$}
	\label{fig:2ndEquation}
\end{figure}

To show the other equation, and establish the adjointness of $i$ and $i_L$, we add the functor $i_L: \pp \to \bb$ to the diagrams in Figure \ref{fig:UnitOfPsi} (see Figure \ref{fig:2ndEquation}).  The natural transformations of the second diagram of Figure \ref{fig:2ndEquation} are given by
\begin{eqnarray*}
1_{i_L} \bullet \delta &:&  i_L \circ j \circ g \circ g_L \Rightarrow i_L \circ i \circ i_L \circ j\\
1_{i_L} \bullet 1_{i} \bullet \gamma &:& i_L \circ i \Rightarrow i_L \circ i \circ i_L \circ i
\end{eqnarray*}

The modification going from the diagram on the left to the diagram on the right in Figure \ref{fig:UnitOfPsi}  induces a modification $\Lambda$ given by
\begin{eqnarray*}
1_{i_L} \bullet 1_j \bullet \eta_{g}&:& i_L \circ j \Rightarrow i_L \circ j \circ g \circ g_L \\
1 &:& i_L \circ i \Rightarrow i_L \circ i
\end{eqnarray*}
which corresponds to the natural transformation $1 \bullet \eta: i_L \Rightarrow i_L \circ i \circ i_L$.

\begin{figure}
	\centering
	\exa
		$$\xymatrix@+20pt{
\aa \ar[r]^-{f} \ar[d]_{g} \xtwocell[rd]{}\omit{^<0>\alpha} & {\bb} \ar@{-->}[d]^{i} \ar@/^2pc/[rrddd]^{i_L \circ i} \xtwocell[rrddd]{}\omit{<-1>} \\
\cc \ar@{-->}[r]_-{j} \ar@/_2pc/[rrrdd]_{i_L \circ j \circ g \circ g_L} \xtwocell[rrrdd]{}\omit{^<0>} & \pp\ar[rd]|-{i_L} \\
&& {\bb} \ar[rd]|-{i_L \circ i} \\
&&& \bb}$$
  \exb
		$$\xymatrix@+20pt{
\aa \ar[r]^-{f} \ar[d]_{g} \xtwocell[rd]{}\omit{^<0>\alpha} & {\bb} \ar@{-->}[d]^{i} \ar@/^2pc/[rrddd]^{i_L \circ i} \xtwocell[rrddd]{}\omit{=} \\
\cc \ar@{-->}[r]_-{j} \ar@/_2pc/[rrrdd]_{i_L \circ j} \xtwocell[rrrdd]{}\omit{=} & \pp \ar[rd]|-{i_L} \\
&& {\bb} \ar@{=}[rd] \\
&&& \bb}$$
  \exc
	\caption{Diagrams occurring in the definition of the modification $\Omega$}
	\label{fig:PsiAdjoint2}
\end{figure}

Likewise, there is a modification $\Omega$ (going from the 2-natural transformation on the left--hand side of Figure \ref{fig:PsiAdjoint2} to the 2-natural transformation on the right--hand side) given by
\begin{eqnarray*}
(\epsilon \bullet 1) \circ (1 \bullet \delta) &:& i_L \circ j \circ g \circ g_L \Rightarrow i_L \circ j \\
1 &:& i_L \circ i \Rightarrow i_L \circ i
\end{eqnarray*}
which induces the following commutative diagrams
$$\xymatrix@C+20pt@R+5pt{
i_L \circ i \ar@{=>}[r]^-{\gamma \bullet 1_{i_L} \bullet 1_{i}}\ar@{=}[d] & i_L \circ i \circ i_L \circ i \ar@{=>}[d]^{\epsilon \bullet 1_{i_L} \bullet 1_i} & {i_L \circ j \circ g \circ g_L} \ar@{=>}[r]^-{1_{i_L} \bullet \delta} \ar@{=>}[d]^{(\epsilon \bullet 1_{i_L} \bullet 1_j) \circ (1_{i_L} \bullet \delta)} & {i_L \circ i \circ i_L \circ j}\ar@{=>}[d]^{\epsilon \bullet 1_{i_L} \bullet 1_j} \\
i_L \circ i \ar@{=}[r] & {i \circ i_L} & {i_L \circ j} \ar@{=}[r] & {i_L \circ j}}$$
such that $\Omega$ corresponds to the natural transformation $\epsilon \bullet 1_{i_L}: i_L \circ i \circ i_L \rightarrow i_L$.

We claim that $\Omega \circ \Lambda = 1$.  The only nontrivial fact to check is that $\Omega \circ \Lambda$ induces the identity  on $i_L \circ j: \cc \to \bb$.  This follows from
\begin{eqnarray*}
&& (\epsilon \bullet 1_{i_L} \bullet 1_j) \circ (1_{i_L} \bullet \delta) \circ (1_{i_L} \bullet 1_j \bullet \eta_g) \\
&=&(\epsilon \bullet 1_{i_L} \bullet 1_j) \circ (1_{i_L} \bullet ((1_i \bullet \beta) \circ (\alpha \bullet 1_{g_L}))) \circ (1_{i_L} \bullet 1_j \bullet \eta_g) \\
&=& (\epsilon \bullet 1_{i_L} \bullet 1_j) \circ (1_{i_L} \bullet 1_i \bullet \beta) \circ (1_{i_L} \bullet \alpha \bullet 1_{g_L}) \circ (1_{i_L} \bullet 1_j \bullet \eta_g) \\
&=& (\epsilon \bullet \beta) \circ (1_{i_L} \bullet \alpha \bullet 1_{g_L}) \circ (1_{i_L} \bullet 1_j \bullet \eta_g) \\
&=& \beta \circ (\epsilon \bullet 1_f \bullet 1_{g_L}) \circ (1_{i_L} \bullet \alpha \bullet 1_{g_L}) \circ (1_{i_L} \bullet 1_j \bullet \eta_g) \\
&=& \beta \circ (1_f \bullet \epsilon_g \bullet 1_{g_L}) \circ (\beta^{-1} \bullet 1_g \bullet 1_{g_L}) \circ (1_{i_L} \bullet 1_j \bullet \eta_g) \\
&=& \beta \circ (1_f \bullet \epsilon_g \bullet 1_{g_L}) \circ (1_f \bullet 1_{g_L} \bullet \eta_g) \circ \beta^{-1} \\
&=& \beta \circ (1_f \bullet 1_{g_L}) \circ \beta^{-1} \\
&=& 1_{i_L} \bullet 1_j
\end{eqnarray*}
where we have used equations (\ref{equation:Compatibility}) and (\ref{equation:AlphaBetaGamma}).  Since $\Omega \circ \Lambda = 1$, we find that $(\epsilon \bullet 1)\circ(1 \bullet \eta) = 1$ and hence we may conclude that $i_L$ is indeed left adjoint to $i$.  Since $\epsilon$ is invertible, the functor $i$ is fully faithful.

Similarly, one verifies that the functor $i_R$, defined in Figure \ref{fig:RightAdjointOfPsi}, is right adjoint to $i$.

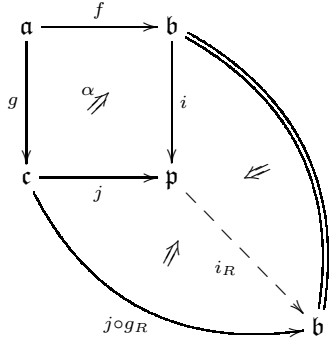
\begin{figure}
	\centering
$$\xymatrix@+20pt{
\aa \ar[r]^-{f} \ar[d]_{g} \xtwocell[rd]{}\omit{^<0>\alpha}& {\bb} \ar[d]^{i} \ar@/^2pc/@{=}[rdd] \xtwocell[rdd]{}\omit{<-1>}\\
\cc \ar[r]_-{j} \ar@/_2pc/[rrd]_{j \circ g_R} \xtwocell[rrd]{}\omit{^<0>}& \pp\ar@{-->}[rd]_{i_R} \\
&& {\bb}}$$
 \caption{Construction of $i_R$}
 \label{fig:RightAdjointOfPsi}
\end{figure}
\end{proof}